\documentclass[11pt,oneside]{article}
\usepackage{tikz}
\usepackage{amssymb,amsmath,latexsym,amsfonts,amscd,amsthm,multirow,colortbl,caption,array}
\usepackage{ctable}
\usepackage{upgreek}
\usepackage{wasysym}
\usepackage{mathrsfs}
\usepackage{fancyhdr,tabularx,cite,mathrsfs,graphicx}
\usepackage{authblk}
\usepackage[headings]{fullpage}
\usepackage{stmaryrd}
\usepackage[all]{xy}
\usepackage{stmaryrd,rotating}
\usepackage{pdflscape,longtable}
\usepackage{accents}
\usepackage[OT2,T1]{fontenc}
\usepackage[font=small]{caption}

\usepackage{setspace}

\usepackage{hyperref}

\newcommand{\bea}{\begin{eqnarray}} 
\newcommand{\eea}{\end{eqnarray}} 
\newcommand{\bee}{\begin{eqnarray*}} 
\newcommand{\eee}{\end{eqnarray*}} 
\newcommand{\al}{\begin{align*}} 
\newcommand{\eal}{\end{align*}} 
\newcommand{\be}{\begin{equation}} 
\newcommand{\ee}{\end{equation}} 
 
\newcommand{\bem}{\begin{pmatrix}} 
\newcommand{\eem}{\end{pmatrix}}

\def\a{\alpha}

\def\f{\phi}  
\def\vf{\varphi}

\def\D{\Delta}

\newcolumntype{R}{ >{$}r <{$}}
\newcolumntype{C}{ >{$}c <{$}}
\newcolumntype{L}{ >{$}l <{$}}
\newcolumntype{F}{>{\centering\arraybackslash}m{1.5cm}}

\newcommand{\mc}[1]{\mathcal{#1}}
\newcommand{\ms}[1]{\mathscr{#1}}

\newcommand{\comment}[1]{}

\DeclareSymbolFont{cyrletters}{OT2}{wncyr}{m}{n}\DeclareMathSymbol{\Sha}{\mathalpha}{cyrletters}{"58}

\newcommand{\RR}{{\mathbb R}}
\newcommand{\CC}{{\mathbb C}}
\newcommand{\PP}{{\mathbb P}}
\newcommand{\ZZ}{{\mathbb Z}}
\newcommand{\QQ}{{\mathbb Q}}
\newcommand{\HH}{{\mathbb H}}
\newcommand{\JJ}{{\mathbb J}}

\newcommand{\End}{\operatorname{End}}

\newcommand{\Irr}{\operatorname{Irr}}

\newcommand{\Span}{\operatorname{Span}}

\newcommand{\tr}{\operatorname{{tr}}}

\newcommand{\num}{\operatorname{num}}
\newcommand{\den}{\operatorname{den}}
\newcommand{\indo}{\iota}

\newcommand{\sk}{{\rm sk}}

\newcommand{\tor}{{\rm tor}}

\newcommand{\opt}{{\rm opt}}
\newcommand{\eop}{{\rm Eis}}

\newcommand{\Eis}{{\rm Eis}}

\newcommand{\xmod}{{\rm \;mod\;}}

\newcommand{\Th}{\Theta}

\newcommand{\Nm}{\operatorname{N}}

\newcommand{\opc}{{\rm c}}

\newcommand{\SL}{\operatorname{\textsl{SL}}}      
\newcommand{\mpt}{\widetilde{\SL}_2}      
\newcommand{\GL}{{\textsl{GL}}}      

\newcommand{\sM}{\mathsf{M}}
\newcommand{\se}{\mathsf{e}}	
\newcommand{\sg}{\mathsf{g}}	

\newcommand{\sZNZ}{{\mathbb{Z}/N\mathbb{Z}}}
\newcommand{\sG}{\mathsf{G}}

\newcommand{\GammaOJ}{\Gamma_0^{\rm J}}

\newcommand{\ON}{\textsf{O'N}}	
\renewcommand{\Th}{\textsf{Th}}		

\newcommand{\HGH}{{H}^{\rm Hur}}	
\newcommand{\HGHO}{{H}^{{\rm Hur},0}}	
\newcommand{\HCE}{{H}^{\rm Coh}}	
\newcommand{\sHGH}{\ms{H}^{\rm Hur}}	
\newcommand{\sHCE}{\ms{H}^{\rm Coh}}	
\newcommand{\sHR}{\ms{H}^{\rm Rad}}	

\newcommand{\dC}{d^{\rm Coh}}
\newcommand{\dH}{d^{\rm Hur}}

\newcommand{\nC}{n^{\rm Coh}}
\newcommand{\nH}{n^{\rm Hur}}

\newtheorem{thm}{Theorem}[subsection]
\newtheorem*{thm*}{Theorem}
\newtheorem{cor}[thm]{Corollary}
\newtheorem{lem}[thm]{Lemma}
\newtheorem{pro}[thm]{Proposition}

\theoremstyle{definition}

\theoremstyle{remark}
\newtheorem{rmk}[thm]{Remark}

\numberwithin{equation}{subsection}

\begin{document}

\setstretch{1.26}

\title{
\vspace{-35pt}
\textsc{\huge{ 
{C}lass {N}umbers, {C}yclic {S}imple {G}roups and {A}rithmetic
}}
    }

\renewcommand{\thefootnote}{\fnsymbol{footnote}} 


\renewcommand{\thefootnote}{\arabic{footnote}} 

\author[1,2]{Miranda C.~N.\ Cheng\thanks{mcheng@uva.nl}\thanks{mcheng@gate.sinica.edu.tw}}
\author[2]{John F.~R.\ Duncan\thanks{jduncan@gate.sinica.edu.tw}}
\author[3]{Michael H.\ Mertens\thanks{mmertens@math.uni-koeln.de}}

\affil[1]{Institute of Physics and Korteweg-de Vries Institute for Mathematics, University of Amsterdam, Amsterdam, the Netherlands.\footnote{On leave from CNRS.}}
\affil[2]{Institute of Mathematics, Academia Sinica, Taipei, Taiwan.}
\affil[3]{Department Mathematik/Informatik, Abteilung Mathematik, Universit\"at zu K\"oln, Weyertal 86--90, D-50931 K\"oln, Germany.}

\date{} 

\maketitle

\abstract{
Here we initiate a program to study relationships between finite groups and arithmetic-geometric invariants in a systematic way. 
To do this we first introduce a notion of optimal module for a finite group in the setting of holomorphic mock Jacobi forms. 
Then we 
classify optimal modules
for the cyclic groups of prime order, in the special case of weight $2$ and index $1$, where class numbers of imaginary quadratic fields play an important role.
Finally we exhibit a connection between the classification we establish and the arithmetic geometry of imaginary quadratic twists of modular curves of prime level.
}

\clearpage

\tableofcontents

\section{Introduction}\label{sec:int}

Relationships between representations of sporadic simple groups and arithmetic-geometric invariants of various kinds have recently appeared in the literature. See \cite{MR4029712,MR4291251,MR4230542} for examples of this, and see \cite{2017NatCo...8..670D,MR4139238} for related expository accounts. 
Our objective here is to initiate a systematic theory of such interrelationships.

\subsection{Motivation}\label{sec:int-mtv}

To motivate our approach we begin by recalling that, in each of the 
abovementioned works, the starting point is reminiscent of monstrous moonshine. More specifically, a (weakly holomorphic) modular form 
$F_\sg$ 
is assigned to each conjugacy class $[\sg]$ of a finite group $\sG$, and it is shown that 
there exists a (virtual) graded $\sG$-module $V=\bigoplus_n V_n$  with the property that
\begin{gather}\label{eqn:int-mtv:FgtrfVn}
F_\sg(\tau)=\sum_n \tr(\sg|V_n)q^n,
\end{gather}
for each $\sg\in {\color{black}\sG}$, where $q=e^{2\pi i \tau}$. Furthermore, the $F_\sg$ are distinguished, in that they satisfy an optimality property that naturally generalizes the principal modulus property of monstrous moonshine{\color{black}:
A principal modulus can be characterized as the unique modular function which has a simple pole at the cusp $\infty$ and nowhere else and $0$ as constant term in its Fourier expansion.  Generalizing this, }
the functions
$F_\sg$
satisfy a growth condition
\begin{gather}\label{eqn:int-mtv:FgP}
	F_\sg(\tau)=P(q^{-1}) +O(q)
\end{gather}
as $\Im(\tau)\to \infty$, for a certain polynomial $P(x)\in\ZZ[x]$ that is independent of $\sg$, and 
also remain bounded near (as many as possible of) the non-infinite cusps of their invariance groups.
(See the cited works for precise formulations.)

For example, in 
\cite{MR4230542}, where $\sG=\Th$ is the sporadic simple group of Thompson, we have $P(x)=6x^5$ in (\ref{eqn:int-mtv:FgP}).
In the setup of \cite{MR4291251}, where $\sG=\ON$ is the (non-monstrous) sporadic simple group of O'Nan, $P(x)=-x^4+2$.
In  monstrous moonshine $P(x)=x$, and optimality, as we have just defined it, is {\color{black}---as mentioned above---} exactly the principal modulus property.

We refer to \cite{Anagiannis:2018jqf,mnstmlts,Duncan2020} for general reviews of moonshine, and refer to \S~6.3 of \cite{MUM}, \S\S~1.5--1.6 of \cite{pmo}, and also \cite{MR4139238}, for more detailed expository discussions of optimality. 

A key difference between monstrous moonshine and the analyses of \cite{MR4029712,MR4291251,MR4230542} is the inclusion of results relating the $\sG$-module $V$ in (\ref{eqn:int-mtv:FgtrfVn}) to arithmetic invariants of elliptic curves in the latter works. For an example of this, let $E$ be the elliptic curve over $\QQ$ defined by $y^2 = x^3 - 12987x -  263466$ (this is Elliptic Curve 15.a5 in \cite{lmfdb}), and for $D$ an integer let $E\otimes D$ denote the $D$-th quadratic twist of $E$, so that 
$E\otimes D$ is defined by 
\begin{gather}\label{eqn:int-mtv:EotimesD}
	y^2 = x^3 - 12987D^2x -  263466D^3.
\end{gather}
Also, recall that $D\in\ZZ$ is 
said to be a {fundamental discriminant} if it coincides with the discriminant of $\QQ(\sqrt{D})$. 
Now, taking $\sG=\ON$, letting $F_\sg$ 
be as in \cite{MR4291251}, and letting $D$ be a negative fundamental discriminant that is $1\xmod 3$, and either $2\xmod 5$ or $3\xmod 5$, it follows from Theorem 1.4 of op.\ cit.\ that 
there are only finitely many rational number solutions to (\ref{eqn:int-mtv:EotimesD}), 
unless 
\begin{gather}\label{eqn:int-mtv:cgDhDequiv0mod5}
\tr(\sg|V_{|D|})+h(D)\equiv 0\xmod 5,
\end{gather} 
where $\sg$ is an element of order $3$ in $\sG$, and $h(D)$ is the class number of $\QQ(\sqrt{D})$. As is illustrated in 
detail in 
\S~7 of \cite{MR4139238}, a basic significance of this 
is that the obstruction (\ref{eqn:int-mtv:cgDhDequiv0mod5}) may be computed in terms of binary quadratic forms, whereas the computation of the rank of an elliptic curve is generally a more difficult task (cf.\ e.g.\ \cite{MR2238272}).

Stepping back for a moment, we may ask if 
such 
interrelationships
(\ref{eqn:int-mtv:EotimesD}--\ref{eqn:int-mtv:cgDhDequiv0mod5})
are isolated exceptional phenomena (like the sporadic simple groups themselves), 
or representative of a broader general theory (like the sporadic simple groups themselves). 
As we explain in more detail, both here and in forthcoming work \cite{radhur,hurm23}, 
the results we have obtained support a balance between both points of view.
To wit, the main results of \cite{radhur,hurm23} are counterparts to (\ref{eqn:int-mtv:FgtrfVn}--\ref{eqn:int-mtv:FgP}) 
which feature the sporadic simple Mathieu groups,
and thereby exemplify the exceptional aspect of the setup in which we work. 
However, we demonstrate here that these sporadic examples are representative of a general theory, by 
producing
an infinite family 
that involves all the cyclic groups of prime order (see Theorem \ref{thm:res-opt:tor}), and specializes, in some cases (see \S~\ref{sec:int-res}), the sporadic examples of \cite{radhur,hurm23}. 
In this work we formulate 
complements to (\ref{eqn:int-mtv:EotimesD}--\ref{eqn:int-mtv:cgDhDequiv0mod5}), for each of our cyclic group counterparts to (\ref{eqn:int-mtv:FgtrfVn}--\ref{eqn:int-mtv:FgP})
(see Theorem \ref{thm:res-art:N=arb} and Corollary \ref{cor:res-art:N=11}),
and we do similarly for sporadic simple Mathieu groups in \cite{radhur,hurm23}. In the latter works we also further extend the arithmetic side of the theory, by 
elucidating some more subtle relationships 
between 
group 
representations 
and arithmetic 
geometric invariants. 

\subsection{Methods}\label{sec:int-mth}

{\color{black} Previous works \cite{MR4291251,MR4230542} use the setting of modular forms of weight $\frac32$.  Through the general theory of Jacobi forms, the relevant aspects of which we recall in Section–\ref{sec:prp-jac}, this setting can be transformed in a natural way to Jacobi forms of weight $2$ and index $1$, which is the setting we choose for this work.} \footnote{As the Shimura correspondence relates modular forms of integer weight to modular forms of half-integer weight, so may the setup of 
\cite{MR4029712} be regarded as related to that of this work.} However, in contrast to 
\cite{MR4291251,MR4230542},
we start with a specific notion of optimality, 
and then seek compatible representations 
of finite simple groups, rather than the other way around.

We are interested in ``where the theory starts'', so we consider the ``simplest possible'' choices for $P$ in (\ref{eqn:int-mtv:FgP}); namely, the constant polynomials. 
This imposes two more points of contrast between our setup and that of \cite{MR4291251,MR4230542}. 
Firstly, it leads us to mock modular forms, since otherwise we are forced to take $F_\se=0$ in (\ref{eqn:int-mtv:FgtrfVn}), for $\se$ {\color{black} denoting} the identity in {\color{black} a given finite group} $\sG$
(and this forces $P=0$ in (\ref{eqn:int-mtv:FgP}) too).
Secondly, it motivates us to strengthen the notion of optimality we described in \S~\ref{sec:int-mtv} (cf.\ (\ref{eqn:int-mtv:FgP})), so as to require vanishing of the $F_\sg$ near the non-infinite cusps of their invariance groups (for otherwise there are ``too many'' choices for $F_\sg$ in (\ref{eqn:int-mtv:FgtrfVn}), as is illustrated below by (\ref{eqn:cln-fgp:phisg})).

Another point of contrast 
is that we actually work with (mock) Jacobi forms of weight $2$, 
rather than (mock) modular forms of weight $\frac32$. This is mainly because it allows us to more easily formulate the precise notions of optimality that we employ, but also because it allows us to attach forms of level $N$ to elements of order $N$ (rather than forms of level $4N$).

Thus, our focus in this work is on 
virtual graded $\sG$-modules $W=\bigoplus_D W_D$ (cf.\ (\ref{eqn:virtualgraded})), for $\sG$ a finite group, with the property that the associated {McKay--Thompson series},
\begin{gather}\label{eqn:int-mth:phiWsg}
	\phi^W_\sg(\tau,z):=\sum_{n,s\in\ZZ}\tr(\sg|W_{s^2-4n})q^ny^s,
\end{gather}
is a mock Jacobi form of weight $2$ and index $1$ for $\GammaOJ(o(\sg))$ (see Proposition \ref{pro:prp-fcs:msHNmock}), for each $\sg\in \sG$, where $o(\sg)$ denotes the order of $\sg$. 
The space of mock Jacobi forms of weight $2$ and index $1$ for $\GammaOJ(1)$ is one-dimensional, spanned by the Hurwitz class number generating function, $\sHGH=\sHGH_1$ (see (\ref{eqn:prp-cln:msH1}) and Proposition \ref{pro:prp-fcs:JJ21NJ21N}). 
So such modules $W$, supposing they exist, exhibit (rescaled) class numbers 
as dimensions of representations of finite groups.

Actually such modules abound, in the absence of further conditions on the McKay--Thompson series (\ref{eqn:int-mth:phiWsg}). 
To see this let $\sG$ be any finite group, and set 
\begin{gather}\label{eqn:cln-fgp:phisg}
\phi_\sg:=12\sHGH+
M\left(\sHGH_{o(\sg)}-
\indo(o(\sg))
\sHGH\right)
\end{gather} 
for $\sg\in\sG$, for a fixed integer $M$, where $\sHGH$ and $\sHGH_N$ are as defined in (\ref{eqn:prp-cln:HHurND}--\ref{eqn:prp-cln:msHN}), and 
$\indo(N)$ is 
the index of $\Gamma_0(N)$ in $\SL_2(\ZZ)$ (cf.\ (\ref{eqn:prp-cln:HN0})).
Then $\phi_\sg$ is a mock Jacobi form of weight $2$ and index $1$ for $\GammaOJ(o(\sg))$ by Proposition \ref{pro:prp-fcs:msHNmock}, and Thompson's characterization of virtual characters (cf.\ e.g.\ \cite{MR822245}) confirms the existence of a virtual graded $\sG$-module $W=\bigoplus_D W_D$ such that $\phi_\sg=\phi^W_\sg$ for each $\sg\in \sG$ (cf.\ (\ref{eqn:int-mth:phiWsg})), so long as $M$ is divisible by sufficiently large powers of the primes that divide $\#\sG$. 
\footnote{Note that the $12$ in (\ref{eqn:cln-fgp:phisg}) 
clears the denominators in the coefficients of $\sHGH$.}

For this reason, and with moonshine as motivation,
we restrict attention in this work to $\sG$-modules $W$ for which the associated mock Jacobi forms $\phi^W_\sg$ are {optimal}, in the sense sketched in \S~\ref{sec:int-mtv}, taking $P$ to be constant in (\ref{eqn:int-mtv:FgP}). That is, we require that
\begin{gather}\label{eqn:int-mth:trv}
\phi^W_\sg(\tau,z)=-\opc+O(q)
\end{gather} 
as $\Im(\tau)\to \infty$, for all $\sg\in\sG$ and $z\in\CC$, for some fixed integer $\opc$, and also require that 
the {theta-coefficients}
\begin{gather}\label{eqn:int-mth:hWgFS}
h^W_{\sg,r}(\tau) := \sum_{D\equiv r^2\xmod 4} \tr(\sg|W_D)q^{-\frac{D}4}
\end{gather}
of $\phi^W_\sg$
tend to $0$
as $\tau$ tends to any cusp of $\GammaOJ(o(\sg))$ other than the infinite one, 
for $r\in \{0,1\}$. 
We formulate this notion of optimality more carefully and more generally in \S~\ref{sec:prp-opt}.

Note that the connection between $\phi^W_\sg$ (\ref{eqn:int-mth:phiWsg}) and the $h^W_{\sg,r}$ (\ref{eqn:int-mth:hWgFS}) is that 
\begin{gather}\label{eqn:int-mth:thetadecomp_phiW}
\phi^W_\sg(\tau,z)= 
h^W_{\sg,0}(\tau)\theta_{1,0}(\tau,z)+h^W_{\sg,1}(\tau)\theta_{1,1}(\tau,z),
\end{gather}
where $\theta_{1,r}$ is as defined in (\ref{eqn:prp-jac:thetadecomp}). Also, the 
weight $\frac32$ (mock) modular form corresponding to $\phi^W_\sg$ is the function
\begin{gather}
\check{h}^W_\sg(\tau):=h^W_{\sg,0}(4\tau)+h^W_{\sg,1}(4\tau),
\end{gather}
so the $\check{h}^W_\sg$ for $\sg\in \sG$ 
are 
more direct 
counterparts to the $F_\sg$ of (\ref{eqn:int-mtv:FgtrfVn}). 

In monstrous, penumbral and umbral moonshine the optimal forms arising are naturally associated to genus zero subgroups of $\SL_2(\RR)$.  (See \cite{pmz,pmo} for more on this in the {\color{black} case of penumbral moonshine}, and see \cite{MR4127159} for a precise formulation and proof in the {\color{black} case of  umbral moonshine}.) So there are only finitely many examples according to \cite{MR604632}. By contrast, optimal mock Jacobi forms of weight $2$ and index $1$ for $\GammaOJ(N)$ exist for all $N$. Indeed, 
we identify an explicit example, $\sHR_N$, for each $N$ in \S~\ref{sec:prp-fcs} (see (\ref{eqn:prp-fcs:msRN}) and Proposition \ref{pro:prp-fcs:msRN_1opt}).
So the relationships to finite groups promise to be richer in the present setting. 

\subsection{Results}\label{sec:int-res}

Say that a virtual graded $\sG$-module $W=\bigoplus_D W_D$ 
as in (\ref{eqn:int-mth:phiWsg}) 
is {optimal} if the associated McKay--Thompson series $\phi^W_\sg$ of (\ref{eqn:int-mth:phiWsg}) are optimal, in the sense of (\ref{eqn:int-mth:trv}--\ref{eqn:int-mth:hWgFS}), for all $\sg\in\sG$.
In a similar spirit to \cite{MR4230542}, we seek an understanding of the full set $\mc{W}_{2,1}^\opt(\sG)$ of optimal $\sG$-modules.  
As we explain in \S~\ref{sec:prp-opt} (see Proposition \ref{pro:prp-opt:mcWG}), this set $\mc{W}_{2,1}^\opt(\sG)$ is naturally a free abelian group of finite rank.
A full solution to the problem of understanding it depends upon 
the minimal positive integer $\opc^\opt_{2,1}(\sG)$ for which 
\begin{gather}\label{eqn:int-res:msubG}
12\opc^\opt_{2,1}(\sG)\sHGH(\tau,z)=-\opc^\opt_{2,1}(\sG)+O(q)
\end{gather}
arises as the graded dimension function $\phi^W_\se$, for some $W\in \mc{W}_{2,1}^\opt(\sG)$ (cf.\ (\ref{eqn:prp-opt:opcopt})), 
and also depends upon a 
lattice structure $\mc{L}^\opt_{2,1}(\sG)$ (cf.\ (\ref{eqn:prp-opt:langleWWprangle}))
on the subgroup of optimal $\sG$-modules that have $\opc=0$ in (\ref{eqn:int-mth:trv}).
With this as motivation we call the computation of $\opc^\opt_{2,1}(\sG)$ and $\mc{L}^\opt_{2,1}(\sG)$, for a given finite group $\sG$, the classification problem for optimal (mock Jacobi) $\sG$-modules (of weight $2$ and index $1$).

In this work we produce a 
solution 
to the
optimal module 
classification problem 
just described, for the ``first'' infinite family of finite simple groups. Specifically, we determine 
$\mc{W}^\opt_{2,1}(\sG)$, for $\sG$ a cyclic group of prime order, by computing $\opc^\opt_{2,1}(\sG)$ and the lattice $\mc{L}^\opt_{2,1}(\sG)$. 
This is the content of our first main result, Theorem \ref{thm:res-opt:tor}. 
Our computation of $\opc^\opt_{2,1}(\sG)$ results in the explicit formula 
\begin{gather}
	\opc^\opt_{2,1}(\sG) = \num\left(\frac{\#\sG+1}{6}\right)
\end{gather}
(see (\ref{eqn:res-opt:opcoptnHN})), where $\num(\a)$ denotes the numerator of a rational number $\a$. 
We compute $\mc{L}^\opt_{2,1}(\sG)$ by expressing it in terms of virtual modules for $\sG$, and the space $S_2(\#\sG)$ 
of 
cuspidal modular forms with weight $2$ and level the order of $\sG$ (see (\ref{eqn:res-opt:mcLoptS2NRG0})). 

Thus, as a consequence of Theorem \ref{thm:res-opt:tor} we see that---even for the simplest family of finite simple groups---structures as rich as spaces of cusp forms play a role in 
classifying its optimal modules.
It is natural to ask if this richness is captured by the structures we introduce in this work. 
Our 
second main result, Theorem \ref{thm:res-art:N=arb},
answers this question affirmatively, 
by 
connecting the classification result of Theorem \ref{thm:res-opt:tor} to the existence, or otherwise, of infinite order rational points on imaginary quadratic twists of modular abelian varieties. 
For example, it is a consequence (see Corollary \ref{cor:res-art:N=11}) of our classification of optimal $\sG$-modules for $\sG=\ZZ/11\ZZ$, that if
$D$ is a negative fundamental discriminant such that $11$ is inert in 
$\QQ(\sqrt{D})$, 
then the $D$-twist 
\begin{gather}\label{eqn:int-res:DtwistJ011}
y^2 = x^3 - 13392D^2x -  1080432D^3	
\end{gather}
(cf.\ (\ref{eqn:int-mtv:EotimesD})), of the modular Jacobian $J_0(11)$ (this is Elliptic Curve 11.a2 in \cite{lmfdb}),
has only finitely many rational number solutions, unless
$h(D)\equiv 0\xmod 5$
(cf.\ (\ref{eqn:int-mtv:cgDhDequiv0mod5})). 
It turns out to be no coincidence (see Theorem \ref{thm:res-art:N=arb}) that $5$ is the number of 
rational points 
of finite order 
on $J_0(11)$.

As we have alluded to in \S~\ref{sec:int-mtv}, the optimal modules for cyclic simple groups that we focus on in this work, in some cases specialize optimal modules for sporadic simple groups that we consider in \cite{radhur,hurm23}. 
For example, it follows from the results of \cite{radhur} that for $\sG=\ZZ/11\ZZ$, every element of $\mc{W}^\opt_{2,1}(\sG)$ 
extends to a module for the unique non-trivial double cover $2.\sM_{12}$ of the Mathieu group $\sM_{12}$. By a similar token, taking $\sG=\ZZ/23\ZZ$ we have that every element of $\mc{W}^\opt_{2,1}(\sG)$ 
extends to a module for the Mathieu group $\sM_{23}$, according to the results of \cite{hurm23}.

As the reader may anticipate, the relative complexity of the groups involved in \cite{radhur,hurm23} entails richer relationships to arithmetic-geometric invariants.
To preview this we mention that
we connect the congruent number problem of antiquity (see e.g.\ the introduction to \cite{MR700775}) to representations of the Mathieu group $\sM_{11}$ in \cite{radhur}, and establish $\sM_{23}$-based interdependencies between the arithmetic-geometric invariants of elliptic curves of different (coprime) levels in \cite{hurm23}.

We conclude by highlighting three problems for future work. 
The first of these is the classification of optimal modules for other finite groups. 
In addition to the sporadic simple groups, which have 
proven particularly successful at producing 
striking results in similar settings, it would be of interest to understand optimal modules for the infinite families of non-abelian finite simple groups. 
A primary motivation for this is the promise of richer relationships to arithmetic geometry, along the lines of those we present in \cite{radhur,hurm23}, but in infinite families rather than in isolated examples.
The classification result of this work may serve as a starting point for this. 
We also point out that we define optimality in some generality in \S~\ref{sec:prp-opt}, in order to prepare for the exploration of other weights and indices.

A second problem is to apply the 
approach of our work, here and in \cite{radhur,hurm23}, to optimality with non-constant $P$ in (\ref{eqn:int-mtv:FgP}).
Thanks to the results of 
\cite{MR4291251,MR4230542}, discussed in \S~\ref{sec:int-mtv}, we know already that interesting examples exist. 
It would be good 
to 
understand where these examples are situated, in the balance between the general and the exceptional that we contrived at the conclusion of \S~\ref{sec:int-mtv},
and it could be profitable to explore 
what more can be said about their 
arithmetic-geometric aspects.

The final problem we emphasize is the construction of richer 
algebraic structure on the optimal modules themselves. 
Do any of the optimal $\sG$-modules we consider admit a $\sG$-invariant algebra, or represent a $\sG$-invariant Lie-type structure of some kind? 
If so, what are the implications of this for the associated arithmetic geometry?

\subsection{Overview}
The structure of this article is as follows. 
We present a guide to the specialized notation that we use in \S~\ref{sec:not}. Then in \S~\ref{sec:prp} we prepare for the statements and proofs of our main results. 
Specifically, we review generalized Hurwitz class numbers, and some related notions, in \S~\ref{sec:prp-cln}, and review (mock) Jacobi forms in \S~\ref{sec:prp-jac}. 
We introduce the notion of optimality we use, in general weight and index, in \S~\ref{sec:prp-opt}, and discuss this notion in more detail in the special case of weight $2$ and index $1$ in \S~\ref{sec:prp-fcs}. 
With the preparation of \S~\ref{sec:prp} in place we formulate and prove our main results in \S~\ref{sec:res}. The formulation and proof of Theorem \ref{thm:res-opt:tor} appears in \S~\ref{sec:res-opt}, and the formulation and proof of Theorem \ref{thm:res-art:N=arb} appears in \S~\ref{sec:res-art}.

\section*{Acknowledgements}

We are grateful to Lea Beneish, Mathew Emerton, Maryam Khaqan, Kimball Martin, Ken Ono, Preston Wake, Eric Zhu and David Zureick-Brown for particularly helpful communication and discussion. We also thank the anonymous referee for their comments which helped to improve an earlier version of this manuscript.

The work of M.C.\ was supported by the National Science and Technology Council of Taiwan (110-2115-M-001-018-MY3), and by a Vidi grant (number 016.Vidi.189.182) from the Dutch Research Council (NWO).
The work of J.D.\ was supported in part by the U.S.\ National Science Foundation (DMS 1601306), the Simons Foundation (\#708354), and the National Science and Technology Council of Taiwan (111-2115-M-001-001-MY2).

\section{Notation}\label{sec:not}

\begin{footnotesize}

\begin{list}{}{
	\itemsep -1pt
	\labelwidth 23ex
	\leftmargin 13ex	
	}

\item
[$(\cdot\,,\cdot)$]
The symmetric bilinear forms on 
$L_{k,m}(\sG)$ and $\mc{L}^\opt_{k,m}(\sG)$. 
See
(\ref{eqn:prp-opt:lambdalambda'})
and 
(\ref{eqn:prp-opt:langleWWprangle}).

\item
[$\langle\cdot\,,\cdot\rangle$]
The Petersson inner product on $S_{k,m}(N)$. 
See (\ref{eqn:prp-jac:Pet}).

\item
[$A\otimes D$]
The $D$-twist of an abelian variety $A$. See (\ref{eqn:res-art:AotimesD}).

\item
[$\mathcal{B}$]
A quaternion algebra over $\QQ$. 
See the proof of Lemma \ref{lem:res-opt:cuspformcong}.

\item
[$\opc$]
The negative of the constant term of a holomorphic mock Jacobi form. See (\ref{eqn:prp-opt:c-optimal}).

\item
[$\opc^\eop(\sG)$]
An invariant we attach to a finite group $\sG$ in \S~\ref{sec:res-opt}.
Cf.\ (\ref{eqn:res-opt:opceopG}).

\item
[$\opc^\eop(N)$]
A certain constant that we define in \S~\ref{sec:res-opt} 
for $N$ prime. See (\ref{eqn:res-opt:opceop}). 

\item
[$\opc^\opt_{k,m}(\sG)$]
An invariant we attach to a finite group $\sG$ in \S~\ref{sec:prp-opt}.
See (\ref{eqn:prp-opt:opcopt}).

\item
[$C_N(D)$]
A shorthand for $C_\f(D)$, for a certain choice of $\f$, in the proof of Theorem \ref{thm:res-art:N=arb}.

\item
[$C_{\vf_N}(D)$]
A coefficient 
in the Fourier expansion of $\varphi_N$. 
Cf.\ (\ref{eqn:res-opt:cuspformcong-msG}).

\item
[$\dC_N$]
The denominator of $\frac{N-1}{12}$ for $N$ prime.
See (\ref{eqn:prp-cln:dHNdCN}).

\item
[$\dH_N$]
The denominator of $\frac{N+1}{6}$ for $N$ prime.
See (\ref{eqn:prp-cln:dHNdCN}).

\item
[$\den(\a)$]
The denominator of a rational number $\a$ when expressed in reduced form. Cf.\ (\ref{eqn:prp-cln:dHNdCN}).

\item
[$\se$]
The identity element of a finite group $\sG$.
See \S~\ref{sec:prp-opt}.

\item
[$e_i$]
A 
point in the geometric fibre of $X_0(N)$ in characteristic $N$. 
See the proof of Lemma \ref{lem:res-opt:cuspformcong}.

\item
[$E_i$]
A supersingular elliptic curve. 
See the proof of Lemma \ref{lem:res-opt:cuspformcong}.

\item
[$f\otimes D$]
The {$D$-th twist} of a cuspidal modular form $f$. 
See (\ref{eqn:res-art:fotimesD}).

\item
[$\mc{F}(N)$]
A fundamental domain for $\Gamma_0(N)$.
Cf.\ (\ref{eqn:prp-jac:Pet}).

\item
[$\phi^W$]
The assignment $\sg\mapsto \phi^W_\sg$ for $W$
a virtual graded $\sG$-module, for $\sG$ a finite group.
Cf.\ (\ref{eqn:prp-opt:phiW}).

\item
[$\phi^W_\sg$]
The McKay--Thompson series 
associated to
the action of 
$\sg$ on 
$W$. 
See (\ref{eqn:prp-opt:phiWsg}) and (\ref{eqn:prp-fcs:phiWsg}).

\item
[$\varphi_N$]
A cuspidal Jacobi form of weight $2$ and index $1$ for $\GammaOJ(N)$. 
See Lemma \ref{lem:res-opt:cuspformcong}. 

\item
[$\sg$]
An element in a finite group $\sG$.
See \S~\ref{sec:prp-opt}.

\item
[$\sG$]
A finite group.
See \S~\ref{sec:prp-opt}.

\item[$\widetilde{\Gamma}_0(N)$]
The metaplectic double cover of $\Gamma_0(N)$. Cf.\ (\ref{eqn:prp-jac:mlttildesl2}).

\item[$\GammaOJ(N)$]
A group of the form $\Gamma_0(N)\ltimes \ZZ^2$. Cf.\ (\ref{eqn:prp-jac:mltGammaOJ}).

\item
[$h^W_\sg$]
The vector-valued function that takes the $h^W_{\sg,r}$ as its components.
See (\ref{eqn:prp-opt:phiWsg}).

\item
[$h^W_{\sg,r}$]
The theta-coefficients of $\phi^W_\sg$.
See (\ref{eqn:prp-opt:hWsgr}).

\item
[$\check h^W_\sg$]
The McKay--Thompson series 
associated to the action of $\sg$ on $\check{W}$. 
See (\ref{eqn:prp-opt:checkhWg}).

\item
[$\HGH(D)$]
The Hurwitz class number of discriminant $D$. 
See (\ref{eqn:prp-cln:msH1}).

\item
[$\HGH_N(D)$]
The generalized Hurwitz class number of level $N$ and discriminant $D$.
See (\ref{eqn:prp-cln:HHurND}--\ref{eqn:prp-cln:HN0}).

\item
[$\HCE_N(D)$]
A coefficient of 
the Cohen--Eisenstein series $\sHCE_N$. 
Cf.\ (\ref{eqn:prp-cln:HCohN}--\ref{eqn:prp-cln:msHCohN}).

\item
[$H^\sZNZ_N(D)$]
A coefficient 
of $\ms{H}^\sZNZ_N$. 
See (\ref{eqn:res-opt:HsNZNND}).

\item
[$\sHGH$]
A shorthand for $\sHGH_1$.
See (\ref{eqn:prp-cln:msH1}).

\item
[$\sHGH_N$]
A holomorphic mock Jacobi form 
defined by the $\HGH_N(D)$.
See (\ref{eqn:prp-cln:msHN}).

\item
[$\sHCE_N$]
A holomorphic Jacobi form 
defined by the $\HCE_N(D)$.
See (\ref{eqn:prp-cln:msHCohN}).

\item
[$\ms{H}^\sZNZ_\sg$]
A certain holomorphic mock Jacobi form defined for $\sg\in \sZNZ$ for $N$ prime. 
See (\ref{eqn:res-opt:msHZNZsg}).

\item
[$\ms{H}^\sZNZ_N$]
A shorthand for $\ms{H}^\sZNZ_{\sg}$ when $N=o(\sg)$.
See the proof of Theorem \ref{thm:res-opt:tor}.

\item
[$\indo(N)$]
The index of $\Gamma_0(N)$ as a subgroup of $\Gamma_0(1)=\SL_2(\ZZ)$. 
Cf.\ (\ref{eqn:prp-cln:HN0}).

\item
[$J_{k,m}(N)$]
The holomorphic Jacobi forms of weight $k$ and index $m$ for $\GammaOJ(N)$. 
Cf.\ (\ref{eqn:prp-jac:varrhomhath_vanishing}).

\item
[$J^\sk_{k,m}(N)$]
The 
skew-holomorphic Jacobi forms of weight $k$ and index $m$ for $\GammaOJ(N)$. 
Cf.\ (\ref{eqn:prp-jac:xi_ses}).

\item
[$\JJ_{k,m}(N)$]
The holomorphic mock Jacobi forms of weight $k$ and index $m$ for $\GammaOJ(N)$. 
Cf.\ (\ref{eqn:prp-jac:varrhomhath_vanishing}).

\item
[$\JJ_{k,m}(N)_\ZZ$]
The $\phi\in \JJ_{k,m}(N)$ with integer coefficients. 
See (\ref{eqn:prp-jac:JJkmGamma0JNZZ}).

\item
[$L_{k,m}(\sG)$]
A certain lattice. See (\ref{eqn:prp-opt:LkmsG}).

\item
[$L_{2,1}(\sG)_0$]
A certain sublattice of $L_{2,1}(\sG)$. See (\ref{eqn:prp-fcs:L21sG0}).

\item
[$\mc{L}^\opt_{k,m}(\sG)$]
The lattice of $0$-optimal virtual graded $\sG$-modules of weight $k$ and index $m$. Cf.\ (\ref{eqn:prp-opt:langleWWprangle}).

\item
[$\lambda_\sg$]
A cuspidal Jacobi form determined by $\lambda\in L_{k,m}(\sG)$ and $\sg\in\sG$. 
See (\ref{eqn:prp-opt:lambdasg}).

\item
[$M_{\frac32}^+(4N)$]
A Kohnen plus space of modular forms. 
Cf.\ (\ref{eqn:prp-jac:checkh}) and 
(\ref{eqn:varthetai}).

\item
[$\nC_N$]
The numerator of $\frac{N-1}{12}$ for $N$ prime.
See (\ref{eqn:res-opt:nHNnCN}).

\item
[$\nH_N$]
The numerator of $\frac{N+1}{6}$ for $N$ prime.
See (\ref{eqn:res-opt:nHNnCN}).

\item
[$\num(\a)$]
The numerator of a rational number $\a$ when expressed in reduced form. Cf.\ (\ref{eqn:res-opt:nHNnCN}).

\item
[$\tilde N$]
An inverse for $N$ modulo $\nC_N$ in the case that $N$ is prime. 
See (\ref{eqn:res-opt:msHZNZsg}).

\item
[$q$]
We set $q=e^{2\pi i \tau}$ for $\tau\in \HH$.

\item[$\mc{Q}_N(D)$]
A set of binary quadratic forms with integer coefficients. 
See \S~\ref{sec:prp-cln}.

\item
[$R_i$]
The endomorphism ring of $E_i$. See 
the proof of Lemma \ref{lem:res-opt:cuspformcong}.

\item
[$R(\sG)$]
The Grothendieck group of finitely generated $\CC\sG$-modules. See (\ref{eqn:prp-opt:RsG}).

\item
[$R(\sG)_0$]
The subgroup of $R(\sG)$ composed of virtual $\sG$-modules $V$ with $\tr(\se|V)=0$. Cf.\ (\ref{eqn:prp-fcs:L21sG0}).

\item
[$\varrho_m$]
A certain unitary representation of $\mpt(\ZZ)$. 
Cf.\ (\ref{eqn:prp-jac:barvarrhomtheta}).

\item
[$S_i$]
A shorthand for $\ZZ+2R_i$. 
See the proof of Lemma \ref{lem:res-opt:cuspformcong}.

\item
[$S_i^0$]
The elements of $S_i$ with vanishing trace.
Cf.\ (\ref{eqn:varthetai}).

\item
[$S_k(N)$]
The cuspidal modular forms of weight $k$ for $\Gamma_0(N)$. 
Cf.\ (\ref{eqn:res-opt:mcLoptS2NRG0}).

\item
[$S_{k,m}(N)$]
The cuspidal Jacobi forms of weight $k$ and index $m$ for $\GammaOJ(N)$. 
See (\ref{eqn:prp-jac:varrhomhath_vanishing}).

\item
[$S_{k,m}(N)_\ZZ$]
The $\phi\in S_{k,m}(N)$ with integer coefficients.
Cf.\ (\ref{eqn:prp-jac:JJkmGamma0JNZZ}).

\item
[$\mpt(\ZZ)$]
The metaplectic double cover of $\SL_2(\ZZ)$. Cf.\ (\ref{eqn:prp-jac:mlttildesl2}).

\item
[$t_d$]
A certain skew-holomorphic mock Jacobi form of theta-type.
Cf.\ (\ref{eqn:prp-fcs:tdk}--\ref{eqn:prp-fcs:t1}).

\item
[${\bf T}_N$]
The level $N$ Hecke algebra.
See the proof of Theorem \ref{thm:res-art:N=arb}.

\item
[$\tau_1$]
The real part of a complex number $\tau\in\HH$.
Cf.\ (\ref{eqn:prp-jac:Pet}).

\item
[$\tau_2$]
The imaginary part of a complex number $\tau\in\HH$.
Cf.\ (\ref{eqn:prp-jac:Pet}).

\item
[$\theta_m$]
The vector-valued function that takes the $\theta_{m,r}$ 
as its components.
Cf.\ (\ref{eqn:prp-jac:thetadecomp_vec}).

\item
[$\theta_{m,r}$]
The theta series defined by the positive-definite even lattices of rank $1$.
See (\ref{eqn:prp-jac:thetamr}).

\item
[$\vartheta$]
A $\ZZ$-linear map on $\mc{X}$. 
See the proof of Lemma \ref{lem:res-opt:cuspformcong}.

\item
[$\vartheta_i$]
The theta series associated to $S_i^0$. 
See (\ref{eqn:varthetai}).

\item
[$w_i$]
Half the number of invertible elements of $R_i$.
See the proof of Lemma \ref{lem:res-opt:cuspformcong}.

\item
[$\mc{W}^\opt_{k,m}{(\sG)}$]
The set of optimal virtual graded $\sG$-modules of weight $k$ and index $m$. 
See (\ref{eqn:prp-opt:mcWoptkmsG}).

\item
[$\mc{W}^\opt_{k,m}{(\sG)}_\opc$]
The set of $\opc$-optimal virtual graded $\sG$-modules of weight $k$ and index $m$. 
See (\ref{eqn:prp-opt:mcWoptkmsG}).

\item
[$x_{\rm E}$]
A certain element of $\mc{X}$. See (\ref{eqn:res-opt:xE}).

\item
[$\mathcal{X}$]
A free $\ZZ$-module 
associated to $X_0(N)$, for $N$ prime.
See the proof of Lemma \ref{lem:res-opt:cuspformcong}.

\item
[$\xi$]
The shadow map $\JJ_{k,m}(N)\to J^\sk_{3-k,m}(N)$. 
See (\ref{eqn:prp-jac:xi_ses}).

\item
[$y$]
We set $y=e^{2\pi i z}$ for $z\in\CC$.

\end{list}

\end{footnotesize}

\section{Preparation}\label{sec:prp}

We prepare for the main arguments of this paper in this section. 
We discuss some variations of class numbers of imaginary quadratic fields, and the relations between them, in \S~\ref{sec:prp-cln}.
Then we formulate our conventions for Jacobi forms in \S~\ref{sec:prp-jac}. 
The notion of optimality is a cornerstone of this work, and we discuss it in detail in \S~\ref{sec:prp-opt}. 
With future applications in mind we work with Jacobi forms of general weight and index in \S\S~\ref{sec:prp-jac}--\ref{sec:prp-opt}. 
In \S~\ref{sec:prp-fcs} we specialize to the situation of main interest in this work, wherein the weight is $2$ and the index is $1$, and in this setting we tie together the topics of the preceding sections, \S\S~\ref{sec:prp-cln}--\ref{sec:prp-opt}.

\subsection{Class Numbers}\label{sec:prp-cln}

For integers $N$ and $D$ 
let $\mc{Q}_N(D)$ denote the set of integer coefficient binary quadratic forms $Q(x,y)=Ax^2+Bxy+Cy^2$ of discriminant $D:=B^2-4AC$ with 
$A\equiv 0 \xmod N$.
Then the group 
\begin{gather}\label{eqn:prp-cln:Gamma0N}
\Gamma_0(N) :=
\left.\left\{ 
\left(\begin{smallmatrix} a&b\\c&d\end{smallmatrix}\right)\in\SL_2(\ZZ)\;\right|\; c\equiv0\xmod N
\right\}
\end{gather}
acts naturally on $\mc{Q}_N(D)$ via the rule
\begin{gather}\label{eqn:prp-cln:Qslashgamma}
\left(Q\left|\left(\begin{smallmatrix}a&b\\c&d\end{smallmatrix}\right)\right.\right)(x,y):=Q(ax+by,cx+dy),
\end{gather} 
so it is natural to consider the 
cardinalities of the orbit spaces $\mc{Q}_N(D)/\Gamma_0(N)$.
At least in the positive-definite case, whereby $D<0$, it turns out to be even more natural to consider the weighted cardinalities, with weights determined by the stabilizers 
\begin{gather}\label{eqn:prp-cln:GammaONQ}
\Gamma_0(N)_Q:=\left\{\gamma\in \Gamma_0(N)\mid Q|\gamma=Q\right\}.
\end{gather}
For $D<0$ such that $\mc{Q}_N(D)$ is not empty
the corresponding weighted cardinality
\begin{gather}\label{eqn:prp-cln:HHurND}
\HGH_N(D):=\sum_{Q\in\mc{Q}_N(D)/\Gamma_0(N)}\frac1{\#\Gamma_0(N)_Q}
\end{gather}
is called 
a {\em generalized Hurwitz class number} of $D$.

Our main motivation for the definition (\ref{eqn:prp-cln:HHurND}) is that it manifests functions with (mock) modular properties (cf.\ \S~\ref{sec:prp-jac}).
To state this more concretely we set $\HGH_N(D):=0$ in case $D<0$ {\color{black} is such that} $\mc{Q}_N(D)$ is empty, {\color{black} we} define 
\begin{gather}\label{eqn:prp-cln:HN0}
\HGH_N(0):=-\frac{1}{12}
\indo(N)
\end{gather}
where $\indo(N):=[\Gamma_0(1):\Gamma_0(N)]$ denotes the index of $\Gamma_0(N)$ in the modular group $\Gamma_0(1)=\SL_2(\ZZ)$, 
and set $\HGH_N(D):=0$ when $D>0$. 
With these conventions the
generating functions
\begin{gather}\label{eqn:prp-cln:msHN}
\sHGH_N(\tau,z):=\sum_{n,s\in\ZZ}\HGH_N(s^2-4n)q^ny^s
\end{gather}
define mock Jacobi forms (see Proposition \ref{pro:prp-fcs:msHNmock}) once we substitute $q=e^{2\pi i\tau}$ and $y=e^{2\pi i z}$, for $\tau\in \HH$ and $z\in \CC$.

In what follows we usually suppress the subscript $N$ from notation in (\ref{eqn:prp-cln:HHurND}--\ref{eqn:prp-cln:msHN}) when $N=1$, so that
\begin{gather}\label{eqn:prp-cln:msH1}
\HGH(D):=\HGH_1(D),\quad
\sHGH:=\sHGH_1.
\end{gather}
The $\HGH(D)$ are called simply {\em Hurwitz class numbers}. 

Say that $D\in\ZZ$ is a {\em discriminant} if $\mc{Q}(D)$ is not empty (i.e. $D$ is congruent to $0$ or $1$ modulo $4$), and say that  $D$ is a {\em fundamental discriminant} if it is the discriminant of the number field $\QQ(\sqrt{D})$. Then for $D$ a negative fundamental discriminant we have $\HGH(D)=\frac{1}{w_D}h(D)$, where $h(D)$ is the usual class number of the imaginary quadratic field $\QQ(\sqrt{D})$, and $w_D$ is half the number of units in the ring of integers of $\QQ(\sqrt{D})$. So $\HGH(-3)=\frac13$, $\HGH(-4)=\frac12$, and $\HGH(D)=h(D)$ for fundamental $D<-4$.

We will also make use of a variation on the class numbers $h(D)$ due to Cohen (cf.\ \cite{MR0382192}). To define this note that for $D$ an arbitrary discriminant we have $D=f^2D_0$, where $D_0$ is the discriminant of $\QQ(\sqrt{D})$ and $f$ is the {\em conductor} of the order 
\begin{gather}
\mathcal{O}_D:=\ZZ\left[\frac{D+\sqrt{D}}2\right].
\end{gather} 
Next, for $N$ prime and $D=f^2D_0$ as above
let $f'$ be the largest factor of $f$ that is coprime to $N$, and set $D':=(f')^2D_0$. 
Then, following Gross (see \S~1 of \cite{MR894322}), we define $\HCE_N(D)$ for $N$ prime and $D$ a negative discriminant
by setting 
\begin{gather}
\HCE_N(D):=
\begin{cases}\label{eqn:prp-cln:HCohN}
	0&\text{ if $N$ splits in $\mathcal{O}_{D'}$,}\\
	\frac12\HGH(D')&\text{ if $N$ is ramified in $\mathcal{O}_{D'}$, }\\ 
	\HGH(D')&\text{ if $N$ is inert in $\mathcal{O}_{D'}$.}
\end{cases}
\end{gather}
We also set $\HCE_N(0):=\frac{N-1}{24}$, and set $\HCE_N(D):=0$ when $D$ is positive, or a negative integer that is not a discriminant. 
We then define the {\em Cohen--Eisenstein series} of (prime) level $N$ by setting
\begin{gather}\label{eqn:prp-cln:msHCohN}
	\sHCE_N(\tau,z):=\sum_{n,s\in\ZZ} \HCE_N(s^2-4n)q^{n}y^s,
\end{gather}
where $q=e^{2\pi i \tau}$ and $y=e^{2\pi i z}$, as in (\ref{eqn:prp-cln:msHN}).

It develops that the Cohen--Eisenstein series (\ref{eqn:prp-cln:msHCohN}) are linearly related to the generalized Hurwitz class number generating functions (\ref{eqn:prp-cln:msHN}), 
and are 
in fact modular (see Proposition \ref{pro:prp-fcs:msHNmock}). 
\begin{lem}\label{lem:cln-ces:ceshcn}
For $N$ prime we have $\sHGH=\sHCE_N+\frac12\sHGH_N$. 
\end{lem}
\begin{proof}
It follows directly from the definitions that the constant terms 
on either side of the desired identity coincide, 
so we must show that $\HGH(D)=\HCE_N(D)+\frac12 \HGH_N(D)$ for $D$ a negative discriminant.  

To begin we observe that the definition (\ref{eqn:prp-cln:HCohN}) of $\HCE_N(D)$ 
may be written more succinctly as 
\begin{gather}\label{eqn:HsubNDHDprime}
\HCE_N(D)=\frac12\left(1-\left(\frac{D'}{N}\right)\right)\HGH(D'),
\end{gather} 
where $\left(\frac{\cdot}{\cdot}\right)$ is the Kronecker symbol (see e.g.\ p.\ 503 of \cite{MR909238} for the definition) 
and $D'$ is as in (\ref{eqn:prp-cln:HCohN}). We will derive a similar expression for $\HGH_N(D)$, by using the fact that 
\begin{gather}
	\HGH_N(D)=\sum_{\ell\in L(D)}\HGHO_N(\tfrac{D}{\ell^2}),
\end{gather}
for any positive integer $N$,
where $L(D):=\{\ell>0\mid \ell^2|D\}$ and $\HGHO_N(D)$ is defined just as $\HGH_N(D)$ is, but replacing $\mathcal{Q}_N(D)$ with the subset of {\em primitive} quadratic forms $\mathcal{Q}_N^0(D):=\left\{Q\in \mathcal{Q}_N(D)\mid (A,B,C)=1\right\}$ in (\ref{eqn:prp-cln:HHurND}). Note that, according to p.~507 of \cite{MR909238}, we have 
\begin{gather}\label{eqn:cardQN0DmodGamma0N}
	\#\mathcal{Q}_N^0(D)/\Gamma_0(N)=\left(n_D(N)+n_{\frac{D}{N^2}}(1)\right)\#\mathcal{Q}_1^0(D)/\Gamma_0(1)
\end{gather}
for $N$ prime, where $n_a(m):=\#\left\{x\in \ZZ/2m\ZZ\mid x^2\equiv a \xmod 4m\right\}$ for $a,m\in \ZZ$ (and $n_a(m):=0$ if $a$ is not an integer). 
Also, by applying Satz 2 in \S~8 of \cite{MR631688} we see that 
for $Q\in \mathcal{Q}^0_N(D)$
the quantity $\#\Gamma_0(N)_Q$ depends only on $D$,  
and furthermore 
\begin{gather}\label{eqn:prp-cln:hashGammaONQ}
	\#\Gamma_0(N)_Q=
	\begin{cases}
		6&\text{ if $D=-3$,}\\
		4&\text{ if $D=-4$,}\\
		2&\text{ if $D<-4$.}
	\end{cases}
\end{gather}
We now put (\ref{eqn:cardQN0DmodGamma0N}) and (\ref{eqn:prp-cln:hashGammaONQ}) together, and obtain that
\begin{gather}\label{eqn:prp-cln:HHurNDsumH0Dnblahnblah}
	\HGH_N(D)=\sum_{\ell\in L(D)}\HGHO_1(\tfrac{D}{\ell^2})\left(n_{\frac{D}{\ell^2}}(N)+n_{\frac{D}{(N\ell)^2}}(1)\right).
\end{gather}

Observe that $n_D(N)=1+\left(\frac{D}{N}\right)$ for $N$ prime. {\color{black}Writing} $D=N^{2a}D'$ where $D'$ is as in (\ref{eqn:prp-cln:HCohN}), {\color{black}we} define $L'(D):=\left\{\ell\in L(D)\mid N^a|\ell\right\}${\color{black}. Then the map 
$L(D')\to L'(D),\ \ell'\mapsto N^a\ell'$
is a bijection and we obtain} $(\frac{D/\ell^{2}}{N})\HGHO_1(\frac{D}{\ell^2})=(\frac{D'}{N})\HGHO_1(\frac{D'}{(\ell')^2})$ for $\ell=N^a\ell'\in L'(D)$. 
The Kronecker symbol 
$(\frac{D/\ell^2}{N})$ vanishes for $\ell\in L(D)\setminus L'(D)$, so we have
\begin{gather}\label{eqn:sumH0Dell2nblah}
\sum_{\ell\in L(D)}\HGHO_1(\tfrac{D}{\ell^2})n_{\frac{D}{\ell^2}}(N)=\HGH(D)+\left(\frac{D'}{N}\right)\HGH(D').
\end{gather}
Next observe that $\ell\mapsto n_{\frac{D}{(N\ell)^2}}(1)$ 
is 
$0$ or $1$ according as $\ell\in L(D)$ belongs to $L'(D)$ or not,
so using $\HGHO_1(\frac{D}{\ell^2})=\HGHO_1(\frac{D'}{(\ell')^2})$ for $\ell=N^a\ell'\in L'(D)$ we obtain
\begin{gather}\label{eqn:HDLprimeDLDnblah}
	\sum_{\ell\in L(D)}\HGHO_1(\tfrac{D}{\ell^2})n_{\frac{D}{(N\ell)^2}}(1)=\HGH(D)-\HGH(D').
\end{gather}
We substitute (\ref{eqn:sumH0Dell2nblah}) and (\ref{eqn:HDLprimeDLDnblah}) into (\ref{eqn:prp-cln:HHurNDsumH0Dnblahnblah}) now, and arrive at the identity
\begin{gather}\label{eqn:HsupNDHDHDprime}
\HGH_N(D)=2\HGH(D)-\left(1-\left(\frac{D'}{N}\right)\right)\HGH(D').
\end{gather}
This, taken together with (\ref{eqn:HsubNDHDprime}), 
verifies that $\HCE_N(D)+\frac12\HGH_N(D)=\HGH(D)$ for $D<0$. 
\end{proof}
For application later, in \S~\ref{sec:res-opt}, we record here the following special cases of (\ref{eqn:HsubNDHDprime}) and (\ref{eqn:HsupNDHDHDprime}).
\begin{lem}\label{lem:HCohNDHNDHD}
For $N$ prime and $D<0$ fundamental we have
\begin{gather}\label{eqn:HCohNDHNDHD}
	\HGH_N(D)=\left(1+\left(\frac DN\right)\right)\HGH(D)
	,\quad
	\HCE_N(D)=\frac12\left(1-\left(\frac DN\right)\right)\HGH(D)
	.
\end{gather}
\end{lem}

Observe that the Hurwitz class number $\HGH(D)=\HGH_1(D)$ (cf.\ (\ref{eqn:prp-cln:HHurND})) is an integer unless $D=-3e^2$ or $D=-4e^2$ for some $e\in\ZZ$. 
In the former case the $\SL_2(\ZZ)$-orbit containing $ex^2+exy+ey^2\in \mc{Q}_D$ contributes $\frac13$ to $\HGH(D)$, and in the latter case the orbit of $ex^2+ey^2$ contributes $\frac12$, and all other orbits make integer contributions. So in particular $6\HGH(D)\in \ZZ$ for every negative discriminant $D$. 
Our last objective in this section is 
the determination of analogous statements for $\HGH_N(D)$ and $\HCE_N(D)$ for $N$ prime.
For this define
\begin{gather}\label{eqn:prp-cln:dHNdCN}
	\dH_N:=\den\left(\frac{N+1}{6}\right),\quad
	\dC_N:=\den\left(\frac{N-1}{12}\right),
\end{gather}
where $\den(\alpha)$ denotes the denominator of a rational number $\a$  when expressed in reduced form.

\begin{lem}\label{lem:cln-ces:cNHNDdNHCND}
For $N$ prime and $D<0$ we have $\dH_N\HGH_N(D)\in \ZZ$ and $\dC_N\HCE_N(D)\in \ZZ$.
\end{lem}
\begin{proof}
The statement that $\dC_N\HCE_N(D)$ is an integer for all $D<0$ can be found in \S~1 of \cite{MR894322}. 
For the integrality of $\dH_N\HGH_N(D)$ we may argue as follows.
For $N=2$ we have $\dH_N=2$, so
\begin{gather}\label{eqn:cls-ces:c2H2D}
	\dH_2\HGH_2(D)=4\HGH(D)-2\left(1-\left(\frac{D'}{2}\right)\right)\HGH(D')
\end{gather}
according to (\ref{eqn:HsupNDHDHDprime}), where $D'$ is as in (\ref{eqn:prp-cln:HCohN}). 
From the remarks preceding the statement of the lemma we have that $\HGH(D)$ and $\HGH(D')$ belong to $\frac12\ZZ$ unless $D=-3e^2$ for some integer $e$. So
the integrality of $\dH_2\HGH_2(D)$ follows from (\ref{eqn:cls-ces:c2H2D}), because if $D=-3e^2$ for some $e\in\ZZ$, then $D'=-3(e')^2$ for some $e'\in\ZZ$, and the right-hand side of (\ref{eqn:cls-ces:c2H2D}) becomes $4(a+\frac13)-4(b+\frac13)=4(a-b)$ for some integers $a$ and $b$. So the claim holds for $N=2$. 

Next consider the case that $N$ is an odd prime that is not $5\xmod 6$. Then $\dH_N=3$ and we have
\begin{gather}\label{eqn:cls-ces:cNHND-Noddnot5mod6}
	\dH_N\HGH_N(D)=6\HGH(D)-3\left(1-\left(\frac{D'}{N}\right)\right)\HGH(D')
\end{gather}
by (\ref{eqn:HsupNDHDHDprime}). 
So we have to check that $1-\left(\frac{D'}N\right)$ is even when $D'=-4(e')^2$ for some $e'$. 
This holds because $\left(\frac{-4}N\right)$ is not $0$ for odd $N$.

Finally suppose that $N\equiv 5\xmod 6$. Then $\dH_N=1$ and we have
\begin{gather}\label{eqn:cls-ces:cNHND-N5mod6}
	\dH_N\HGH_N(D)=2\HGH(D)-\left(1-\left(\frac{D'}{N}\right)\right)\HGH(D').
\end{gather}
We also have $\left(\frac{-3}{N}\right)=-1$, so if $D=-3e^2$ for some $e$ then the right-hand side of (\ref{eqn:cls-ces:cNHND-N5mod6}) becomes $2(a+\frac13)-2(b+\frac13)=2(a-b)$ for some integers $a$ and $b$. If $D=-4e^2$ for some $e$ then $\left(\frac{-4}N\right)$ is not zero so $1-\left(\frac{D'}N\right)$ is even. So the integrality of $\dH_N\HGH_N(D)$ follows from (\ref{eqn:cls-ces:cNHND-N5mod6}) in this case too.
\end{proof}

\subsection{Jacobi Forms}\label{sec:prp-jac}

Here we explain our conventions for mock Jacobi forms. 
For this we assume some familiarity with the basic definitions. 
We refer to \S~3.1 of \cite{MR4127159} and the classic text \cite{MR781735} for background on Jacobi forms, and refer to
\S~3.2 of \cite{MR4127159} and \S~7.2 of \cite{Dabholkar:2012nd} for more on mock Jacobi forms. 

For $N$ a positive integer let 
$\GammaOJ(N)$ denote the 
group composed
of the pairs $(\gamma,(\lambda,\mu))$, with $\gamma\in \Gamma_0(N)$ (see (\ref{eqn:prp-cln:Gamma0N})) and $(\lambda,\mu)\in\ZZ^2$, with multiplication given by 
\begin{gather}\label{eqn:prp-jac:mltGammaOJ}
(\gamma,(\lambda,\mu))(\gamma',(\lambda',\mu'))
=
\left(\gamma\gamma',(\lambda,\mu)\gamma'+(\lambda',\mu')\right).
\end{gather}
Then $\GammaOJ(N)$ takes the form $\GammaOJ(N)=\Gamma_0(N)\ltimes \ZZ^2$, and in particular $\GammaOJ(1)=\SL_2(\ZZ)\ltimes \ZZ^2$.

We will also make use of the {\em metaplectic double cover} of $\SL_2(\ZZ)$, denoted $\widetilde{\SL}_2(\ZZ)$, which we realize as the set of pairs $(\gamma,\upsilon)$, where $\gamma\in \SL_2(\ZZ)$, and $\upsilon:\HH\to\CC$ is either of the two smooth functions such that $\upsilon(\tau)^2=c\tau+d$ when $(c,d)$ is the lower row of $\gamma$. The multiplication in this case is given by
\begin{gather}\label{eqn:prp-jac:mlttildesl2}
	(\gamma,\upsilon)(\gamma',\upsilon') = (\gamma\gamma',(\upsilon\circ \gamma') \upsilon').
\end{gather}
We write $\widetilde{\Gamma}_0(N)$ for the preimage of $\Gamma_0(N)$ in $\widetilde{\SL}_2(\ZZ)$. 

The
action of $\SL_2(\ZZ)$ on $\HH$ extends naturally to 
a 
transitive 
action
on the projective line $\PP^1(\QQ)=\QQ\cup\{\infty\}$ over $\QQ$. 
So 
we obtain an action of $\GammaOJ(1)$ on $\PP^1(\QQ)$ by letting the normal subgroup $\{(I,(\lambda,\mu))\}=\ZZ^2$ (cf.\ (\ref{eqn:prp-jac:mltGammaOJ})) act trivially, and then obtain an action of $\GammaOJ(N)$, for any $N$, by restriction.
Thus 
we may consider the set 
\begin{gather}\label{eqn:prp-jac:cuspsofGamma}
	\GammaOJ(N)\backslash\PP^1(\QQ) = \{\GammaOJ(N)\cdot\a\mid \a\in\QQ\cup\{\infty\}\}
\end{gather}
of
orbits of $\GammaOJ(N)$ on $\PP^1(\QQ)$. We call these orbits (\ref{eqn:prp-jac:cuspsofGamma}) the {\em cusps} of $\GammaOJ(N)$, and we refer to the orbit $\GammaOJ(N)\cdot\infty$ containing $\infty$ as the {\em infinite cusp} of $\GammaOJ(N)$.

For any positive integer $N$, a mock Jacobi form of weight $k$ and positive integer index $m$ for $\GammaOJ(N)$ 
admits a {\em theta-decomposition}
\begin{gather}\label{eqn:prp-jac:thetadecomp}
	\phi(\tau,z)
	= \sum_{r\xmod 2m} h_r(\tau)\theta_{m,r}(\tau,z)
\end{gather}
(cf.\ (\ref{eqn:int-mth:thetadecomp_phiW})),
where the theta series $\theta_{m,r}$, standard in the theory, are defined for integers $m$ and $r$, with $m$ positive, by setting
\begin{gather}\label{eqn:prp-jac:thetamr}
	\theta_{m,r}(\tau,z):=\sum_{s\equiv r\xmod 2m}q^{\frac{s^2}{4m}}y^s.
\end{gather}
Moreover, 
the functions $h_r$ in (\ref{eqn:prp-jac:thetadecomp}), called the {\em theta-coefficients} of $\phi$, 
admit {\em Fourier series} expansions of the form
\begin{gather}\label{eqn:prp-jac:hr}
	h_r(\tau)=\sum_{D\equiv r^2\xmod 4m} C_\phi(D,r)q^{-\frac{D}{4m}}.
\end{gather}

In (\ref{eqn:prp-jac:thetamr}--\ref{eqn:prp-jac:hr}), 
and throughout this work, we take $q=e^{2\pi i \tau}$ and $y=e^{2\pi i z}$ for $\tau\in\HH$ and $z\in\CC$
(cf.\ (\ref{eqn:prp-cln:msHN}), (\ref{eqn:prp-cln:msHCohN})),
and we only consider Jacobi forms of integer weight and positive integer index. 
(See \cite{MR3995918} for a discussion of mock Jacobi forms of half-integer index, along with applications to the module problem in umbral moonshine). 
Also, we write the theta-decomposition (\ref{eqn:prp-jac:thetadecomp}) compactly as
\begin{gather}\label{eqn:prp-jac:thetadecomp_vec}
	\phi(\tau,z) = h(\tau)^{\rm t}\theta_m(\tau,z)
\end{gather}
or even $\phi=h^{\rm t}\theta_m$
when convenient, taking $h=(h_r)$ to be the vector-valued function with the theta-coefficients $h_r$ (\ref{eqn:prp-jac:hr}) as its components, and taking $\theta_m=(\theta_{m,r})$ to be the vector-valued function whose components are the theta series $\theta_{m,r}$ (\ref{eqn:prp-jac:thetamr}). \footnote{The superscript in $h(\tau)^{\rm t}$ and $h^{\rm t}$ denotes matrix transposition.}

Note that there is redundancy in the theta-decomposition (\ref{eqn:prp-jac:thetadecomp}), 
because $\theta_{m,r}(\tau,-z)=\theta_{m,-r}(\tau,z)$ (cf.\ (\ref{eqn:prp-jac:thetamr})). This manifests in the rule that
\begin{gather}\label{eqn:prp-jac:hminusr}
	h_{r}(\tau) = (-1)^k h_{-r}(\tau)
\end{gather}
at the level of theta-coefficients (\ref{eqn:prp-jac:hr}), because $\phi(\tau,-z)=(-1)^k\phi(\tau,z)$ when $\phi$ is a mock Jacobi form of integer weight $k$ (by invariance under the action of $(-I,(0,0))\in \GammaOJ(1)$, cf.\ (\ref{eqn:prp-jac:slashkm})).

Especially when expressed in the form (\ref{eqn:prp-jac:thetadecomp_vec}), the 
theta-decomposition 
evidences a relationship between mock Jacobi forms of integer weight and vector-valued mock modular forms of half-integer weight that will be useful for us in what follows. To formulate this relationship precisely 
we first 
define 
the slash operator $\vf\mapsto \vf|_{k,m}(\gamma,\upsilon)$ 
on (vector-valued) functions on $\HH\times\CC$, for $k\in\frac12\ZZ$ and $m\in \ZZ^+$, and for $(\gamma,\upsilon)\in\widetilde{\SL}_2(\ZZ)$ (cf.\ (\ref{eqn:prp-jac:mlttildesl2})), by setting
\begin{gather}\label{eqn:prp-jac:slashkm}
	(\vf|_{k,m}(\gamma,\upsilon))(\tau,z) := \vf\left(\frac{a\tau+b}{c\tau+d},\frac{z}{c\tau+d}\right)\frac1{\upsilon(\tau)^{2k}}\exp\left(-2\pi i\frac{cmz^2}{c\tau+d}\right)
\end{gather}
in case $\gamma=\left(\begin{smallmatrix}a&b\\c&d\end{smallmatrix}\right)$.
Note that $\vf|_{k,m}(\gamma,\upsilon)$, for either choice of $\upsilon$, recovers the usual weight $k$ and index $m$ action of $(\gamma,(0,0))\in \GammaOJ(1)$, when $k$ is an integer.
Next we 
recall (see e.g.\ \S~3.1 of \cite{MR4127159})
that 
we may define a unitary representation $\varrho_m:\widetilde{\SL}_2(\ZZ)\to \GL_{2m}(\CC)$ 
by requiring that
\begin{gather}\label{eqn:prp-jac:barvarrhomtheta}
	\overline{\varrho_m(\gamma,\upsilon)} \theta_m |_{\frac12,m} (\gamma,\upsilon) = \theta_m
\end{gather}
for $(\gamma,\upsilon) \in \widetilde{\SL}_2(\ZZ)$, where
$\theta_m=(\theta_{m,r})$ is as in (\ref{eqn:prp-jac:thetadecomp_vec}).

Now let $\phi$ be a mock Jacobi form of weight $k$ and index $m$ for $\GammaOJ(N)$, 
and write $\hat\phi$ for the modular completion of $\phi$. Then a theta-decomposition 
\begin{gather}\label{eqn:prp-jac:thetadecomp_hatphi}
	\hat\phi(\tau,z) = \hat h(\tau)^{\rm t}\theta_{m}(\tau,z)
\end{gather}
of the form 
(\ref{eqn:prp-jac:thetadecomp_vec})
holds for $\hat\phi$, 
where the components $\hat h_r$ of 
the vector-valued function 
$\hat h=(\hat h_r)$, 
being the theta-coefficients of $\hat\phi$,
are the modular completions of the theta-coefficients of $\phi$.
From (\ref{eqn:prp-jac:barvarrhomtheta}--\ref{eqn:prp-jac:thetadecomp_hatphi}) and the invariance of $\hat \phi$ under the usual weight $k$ and index $m$ action of $\GammaOJ(N)$ we then obtain that
\begin{gather}\label{eqn:prp-jac:varrhomhath}
	{\varrho_m(\gamma,\upsilon)} \hat h |_{k-\frac12} (\gamma,\upsilon) = \hat h
\end{gather}
for $(\gamma,\upsilon) \in \widetilde{\Gamma}_0(N)$, where 
the action $f\mapsto f|_{k}(\gamma,\upsilon)$, on (vector-valued) functions on $\HH$, is given by
\begin{gather}\label{eqn:prp-jac:slashk}
	(f|_{k}(\gamma,\upsilon))(\tau) := f\left(\frac{a\tau+b}{c\tau+d}\right)\frac1{\upsilon(\tau)^{2k}}
\end{gather}
for $k\in \frac12\ZZ$ and $(\gamma,\upsilon)\in \widetilde{\SL}_2(\ZZ)$,
when $\gamma=\left(\begin{smallmatrix}a&b\\c&d\end{smallmatrix}\right)$. 

We call $\varrho_m$ the {\em Weil representation} of $\widetilde{\SL}_2(\ZZ)$ of {\em index $m$}, and we interpret 
(\ref{eqn:prp-jac:varrhomhath}) as saying that 
$\hat h=(\hat h_r)$ is a (real analytic) vector-valued modular form of weight $k-\frac12$ for the restriction of the Weil representation of index $m$ to $\widetilde{\Gamma}_0(N)$.
 
Before returning our focus to Jacobi forms we mention that, with a closer analysis of the Weil representation $\varrho_m$ (cf.\ (\ref{eqn:prp-jac:barvarrhomtheta})), it may be shown that if 
$\phi$ 
is a mock Jacobi form 
of weight $k$ and index $m$ for $\GammaOJ(N)$ for some $N$,
with theta-coefficients $h_r$ (cf.\ (\ref{eqn:prp-jac:thetadecomp})), 
then
\begin{gather}\label{eqn:prp-jac:checkh}
	\check{h}(\tau):=\sum_{r\xmod 2m} h_r(4m\tau)=\sum_{r\xmod 2m}\sum_{D\equiv r^2\xmod 4m}C_\phi(D,r)q^{-D}
\end{gather}
is a mock modular form in the Kohnen plus space \cite{MR575942,MR660784} of weight $k-\frac12$ for $\Gamma_0(4mN)$.
This explains one way in which scalar-valued (mock) modular forms of half-integer weight may stand in for (mock) Jacobi forms of integer weight, and vice-versa.
Note however that 
this construction (\ref{eqn:prp-jac:checkh}) vanishes identically unless $k$ is even, on account of (\ref{eqn:prp-jac:hminusr}).  
Also, it is generally not possible to recover $h=(h_r)$ from $\check{h}$ (cf.\ (\ref{eqn:prp-jac:thetadecomp_vec})), even when $k$ is even, without further assumptions on $\phi$. The exception to this rule is the case that $m$ is not composite (i.e.\ $m$ is $1$ or a prime), for that is the only case where $r^2\equiv s^2\xmod 4m$ implies $r\equiv \pm s\xmod 2m$ for all integers $r$ and $s$, and we have that $C_\phi(D,r)=C_\phi(D,-r)$ when $k$ is even, again by (\ref{eqn:prp-jac:hminusr}).

Typically, in the theory of mock Jacobi forms, the theta-coefficients (\ref{eqn:prp-jac:hr}) are allowed to have exponential growth at cusps. In this work we are specifically interested in mock Jacobi forms that are {\em holomorphic}, in the sense that the 
vector-valued form $\hat h=(\hat h_r)$ 
in (\ref{eqn:prp-jac:thetadecomp_hatphi}) remains bounded as $\Im(\tau)\to \infty$, and 
similarly at non-infinite cusps (cf. (\ref{eqn:prp-jac:cuspsofGamma})). 
With the representation $\varrho_m$ of (\ref{eqn:prp-jac:barvarrhomtheta}) in hand we may formulate this concretely as the condition that
\begin{gather}\label{eqn:prp-jac:varrhomhath_bounded}
	\left({\varrho_m(\gamma,\upsilon)} \hat h |_{k-\frac12} (\gamma,\upsilon)\right)(\tau) = O(1)
\end{gather}
as $\Im(\tau)\to \infty$, for all $(\gamma,\upsilon)\in \widetilde{\SL}_2(\ZZ)$ (cf.\ (\ref{eqn:prp-jac:varrhomhath})).
All the mock Jacobi forms we consider in this work will be holomorphic in this sense. 

For 
$N$ a positive integer
let
$\JJ_{k,m}(N)$ 
denote the space of holomorphic mock Jacobi forms of weight $k$ and index 
$m$ 
for $\GammaOJ(N)$.
The subspace of holomorphic Jacobi forms is composed of the $\phi\in\JJ_{k,m}(N)$ such that $\hat\phi=\phi$ (cf.\ (\ref{eqn:prp-jac:thetadecomp_hatphi})), and we 
denote it $J_{k,m}(N)$.
The subspace of cuspidal Jacobi forms is composed of the $\phi\in J_{k,m}(N)$ such that 
\begin{gather}\label{eqn:prp-jac:varrhomhath_vanishing}
	\left({\varrho_m(\gamma,\upsilon)} h |_{k-\frac12} (\gamma,\upsilon)\right)(\tau) \to 0
\end{gather}
as $\Im(\tau)\to \infty$, for all $(\gamma,\upsilon)\in \widetilde{\SL}_2(\ZZ)$ (cf.\ (\ref{eqn:prp-jac:varrhomhath_bounded})),
and
we 
denote it $S_{k,m}(N)$. 
Also, we use a subscript $\ZZ$ to specify the submodules composed of forms with rational integer Fourier coefficients, so that 
\begin{gather}\label{eqn:prp-jac:JJkmGamma0JNZZ}
\JJ_{k,m}(N)_\ZZ
:=\{
\phi\in \JJ_{k,m}(N) \mid C_\phi(D,r)\in\ZZ\text{ for all $D$ and $r$}
\}.
\end{gather}
The $\ZZ$-modules $J_{k,m}(N)_\ZZ$ and $S_{k,m}(N)_\ZZ$ are defined analogously.

If $\phi$ belongs to $\JJ_{k,m}(N)$ and $N$ is the smallest positive integer for which this statement is true we say that $\phi$ has {\em level} $N$.  

Recall that the space $S_{k,m}(N)$ becomes an inner product space when equipped with the Petersson inner product.
In this work we define this inner product concretely by setting
\begin{gather}\label{eqn:prp-jac:Pet}
	\langle\phi,\phi'\rangle
	:=
	\frac1{\indo(N)}
	\frac{3}{\pi\sqrt{2m}}	
	\sum_{r\xmod 2m}
	\int_{\mc{F}(N)} 
	h_{r}(\tau)
	\overline{h'_{r}(\tau)}
	\tau_2^{k-\frac52}
	{\rm d}\tau_1
	{\rm d}\tau_2
\end{gather}
for $\phi,\phi'\in S_{k,m}(N)$, where $h=(h_r)$ and $h'=(h'_r)$ are the theta-coefficients (cf.\ (\ref{eqn:prp-jac:thetadecomp})) of $\phi$ and $\phi'$, respectively, 
$\indo(N)$ is as in (\ref{eqn:prp-cln:HN0}), 
we take $\mc{F}(N)$ to be a fundamental domain for $\Gamma_0(N)$, 
and we write
$\tau=\tau_1+i\tau_2$ for the decomposition of $\tau$ into its real and imaginary parts.
Note that $\frac{\pi}3\indo(N)$ is the area of $\mc{F}(N)$. 
Scaling by $\indo(N)$ in (\ref{eqn:prp-jac:Pet}) we obtain a normalization of $\langle \phi,\phi'\rangle$ that is independent of $N$, 
so long as $N$ is such that $\phi$ and $\phi'$ both belong to $S_{k,m}(N)$. 

To conclude this section we recall the $\CC$-linear {\em shadow map}, denoted $\xi$, which defines an exact sequence
\begin{gather}\label{eqn:prp-jac:xi_ses}
	0\to J_{k,m}(N)\to \JJ_{k,m}(N)\xrightarrow{\xi} J_{3-k,m}^\sk(N),
\end{gather}
where $J^\sk_{3-k,m}(N)$ is the space of skew-holomorphic Jacobi forms of weight $3-k$ and index $m$ for $\GammaOJ(N)$.
\footnote{We refer to \S~3.1 of \cite{MR4127159} 
for background on skew-holomorphic Jacobi forms, and also recommend 
Skoruppa's works \cite{MR1096975,MR1074485}, where skew-holomorphic Jacobi forms were first introduced.}

\subsection{Optimality}\label{sec:prp-opt}

We have exposited a notion of optimality for holomorphic mock Jacobi forms of weight $2$ and index $1$ in \S~\ref{sec:int-mth}. 
In this section we explain this notion 
more carefully. 
Moreover, with future applications in mind we work in greater generality.

To begin 
we recall 
that a holomorphic mock Jacobi form $\phi \in\JJ_{2,1}(N)$ is called {\em optimal} if its theta-coefficients vanish in the neighborhood of any non-infinite cusp (cf.\ (\ref{eqn:int-mth:trv}--\ref{eqn:int-mth:hWgFS})). We now formulate this concretely, for $\phi\in\JJ_{k,m}(N)$ with arbitrary integer weight $k$ and positive integer index $m$, as the requirement that 
\begin{gather}\label{eqn:prp-opt:optimal}
	\left({\varrho_m(\gamma,\upsilon)} \hat h |_{k-\frac12} (\gamma,\upsilon)\right)(\tau) \to 0
\end{gather}
as $\Im(\tau)\to \infty$ (cf.\ (\ref{eqn:prp-jac:slashk}--\ref{eqn:prp-jac:varrhomhath_bounded})), whenever $(\gamma,\upsilon)\in\widetilde{\SL}_2(\ZZ)$ 
(cf.\ (\ref{eqn:prp-jac:mlttildesl2}))
is such that $\gamma\cdot\infty$ does not belong to the infinite cusp of $\GammaOJ(N)$ (cf.\ (\ref{eqn:prp-jac:cuspsofGamma})). 

For any holomorphic mock Jacobi form $\phi\in\JJ_{k,m}(N)$, and in particular for the optimal ones, 
we must have
\begin{gather}\label{eqn:prp-opt:c-optimal}
	\phi(\tau,z)=-\opc+O(q)
\end{gather}
as $\Im(\tau)\to\infty$, 
for any fixed $z$, 
for some constant $\opc$.
So we may stratify the optimal forms in a space by specifying a particular choice. 
With this in mind we say that $\phi\in \JJ_{k,m}(N)$ is {\em $\opc$-optimal} for a given constant $\opc$, if $\phi$ is optimal (\ref{eqn:prp-opt:optimal}) and satisfies (\ref{eqn:prp-opt:c-optimal}). 

Note that $0$-optimality 
is equivalent to cuspidality (\ref{eqn:prp-jac:varrhomhath_vanishing}), so that if $\phi$ and $\phi'$ are $\opc$-optimal forms in $\JJ_{k,m}(N)$ for some $\opc$ then their difference is cuspidal, $\phi-\phi'\in S_{k,m}(N)$.

Beyond the consideration of mock Jacobi forms on an individual basis, 
we are really interested in collections of mock Jacobi forms that are organized by finite groups. 
To put this
precisely 
suppose that 
$\sG$ is a finite group
and let $R(\sG)$ denote the 
Grothendieck group of the category of finitely generated $\CC\sG$-modules. 
By a {\em virtual} $\sG$-module we mean an element of $R(\sG)$, and by a {\em virtual graded} $\sG$-module we mean an indexed collection 
\begin{gather}\label{eqn:virtualgraded}
V=\bigoplus_{i\in I}V_i
\end{gather} 
of virtual $\sG$-modules $V_i\in R(\sG)$, for some indexing set $I$. 

Let $\Irr(\sG)$ denote the set of irreducible characters of $\sG$. For concreteness we employ the natural identification
\begin{gather}\label{eqn:prp-opt:RsG}
	R(\sG) = \sum_{\chi\in \Irr(\sG)}\ZZ\chi
\end{gather}
of $R(\sG)$ 
with the free $\ZZ$-module 
generated by 
$\Irr(\sG)$ 
in what follows.
Then, for $V\in R(\sG)$ a virtual $\sG$-module we have
\begin{gather}\label{eqn:prp-opt:V=mchiVvhi}
V=\sum_{\chi\in\Irr(\sG)} m_\chi(V)\chi
\end{gather}
for some uniquely determined integers $m_\chi(V)\in \ZZ$. 
Given $V\in R(\sG)$ and $\chi\in \Irr(\sG)$ we call $m_\chi(V)$ as in (\ref{eqn:prp-opt:V=mchiVvhi}) the {\em multiplicity} of $\chi$ in $V$, and given also $\sg\in \sG$ we interpret $\tr(\sg|V)$ as meaning
\begin{gather}\label{eqn:prp-opt:trsgV}
	\tr(\sg|V) = \sum_{\chi\in\Irr(\sG)} m_\chi(V) \chi(\sg).
\end{gather}

Now fix an integer $k$ and a positive integer $m$, and suppose that $W$ is a
virtual graded $\sG$-module 
(cf.\ (\ref{eqn:virtualgraded}))
with grading of the form
\begin{gather}\label{eqn:prp-opt:W}
W=\bigoplus_{r\xmod 2m}\bigoplus_{D\equiv r^2\xmod 4m}W_{r,\frac{D}{4m}}.
\end{gather} 
Then, given $\sg\in\sG$ define 
the associated {\em McKay--Thompson series}, denoted $\phi^W_\sg$, 
by 
requiring that 
\begin{gather}\label{eqn:prp-opt:phiWsg}
\phi^W_\sg(\tau,z)=h^W_\sg(\tau)^{\rm t}\theta_m(\tau,z)
\end{gather} 
(cf.\ (\ref{eqn:prp-jac:thetadecomp_hatphi})), where the components of the vector-valued function $h^W_\sg= (h^W_{\sg,r})$ 
are defined 
by 
setting
\begin{gather}\label{eqn:prp-opt:hWsgr}
	h^W_{\sg,r}(\tau):=\sum_{D\equiv r^2\xmod 4m} 
	\tr\left(\sg\left|W_{r,\frac{D}{4m}}\right.\right)q^{-\frac{D}{4m}}
\end{gather}
(cf.\ (\ref{eqn:prp-opt:trsgV})).
We are interested in the situation that 
$\phi^W_\sg$ is a holomorphic mock Jacobi form of 
weight $k$, 
index $m$ and level $o(\sg)$---and 
moreover satisfies the optimality condition (\ref{eqn:prp-opt:optimal}--\ref{eqn:prp-opt:c-optimal}) just discussed---for each $\sg\in \sG$.

This leads us to a 
notion of optimality for virtual graded $\sG$-modules.
Specifically, given a fixed constant $\opc$ we say that a virtual graded $\sG$-module $W$ as in (\ref{eqn:prp-opt:W}) is {\em $\opc$-optimal (mock Jacobi)} of weight $k$ and index $m$ 
if for all $\sg\in \sG$ we have
\begin{gather}\label{eqn:prp-opt:alpha_condition}
\phi^W_\sg \in
\JJ_{k,m}(N)
\end{gather}
when $N=o(\sg)$, but not for any smaller value of $N$, 
and if $\phi^W_\sg$ also satisfies the $\opc$-optimality conditions 
(\ref{eqn:prp-opt:optimal}--\ref{eqn:prp-opt:c-optimal})
for each $\sg$. 
Also, we call $W$ as in (\ref{eqn:prp-opt:W}) {\em optimal} 
if it is $\opc$-optimal for some $\opc$.

We now 
consider the task of classifying the optimal 
(mock Jacobi) virtual graded $\sG$-modules of given weight and index for a finite group $\sG$. 
For this we let 
$\mc{W}^\opt_{k,m}(\sG)$ denote the set of optimal virtual graded $\sG$-modules of weight $k$ and index $m$,
and given $\opc\in \ZZ$ let $\mc{W}^\opt_{k,m}(\sG)_\opc$ denote the subset of $\opc$-optimal modules. 
If $W\in \mc{W}^\opt_{k,m}{(\sG)}_\opc$ and $W'\in \mc{W}^\opt_{k,m}{(\sG)}_{\opc'}$ for some $\opc,\opc'\in \ZZ$ then $W+W'\in \mc{W}^\opt_{k,m}{(\sG)}_{\opc+\opc'}$, so 
$\mc{W}^\opt_{k,m}{(\sG)}$ and $\mc{W}^\opt_{k,m}{(\sG)}_0$ are free abelian groups, and
we have 
a decomposition
\begin{gather}\label{eqn:prp-opt:mcWoptkmsG}
	\mc{W}^\opt_{k,m}(\sG)
	=
	\sum_{\opc\in\ZZ} \mc{W}^\opt_{k,m}{(\sG)}_\opc,
\end{gather}
of the former into modules for the latter.

To get a better handle on the nature of $\mc{W}^\opt_{k,m}{(\sG)}_0$ we consider the lattice
\begin{gather}\label{eqn:prp-opt:LkmsG}
L_{k,m}(\sG) :=S_{k,m}(\#\sG)_\ZZ\otimes_\ZZ R(\sG)
\end{gather}
(cf.\ (\ref{eqn:prp-jac:JJkmGamma0JNZZ}), (\ref{eqn:prp-opt:RsG})), with bilinear form obtained by combining the Petersson inner product on $S_{k,m}(\#\sG)$ with the usual inner product on $R(\sG)$ {\color{black} (see \eqref{eqn:prp-opt:lambdalambda'} below)}.
To put this concretely we regard an element $\lambda\in L_{k,m}(\sG)$ as defining a $S_{k,m}(\#\sG)$-valued class function $\sg\mapsto \lambda_\sg$ on $\sG$ by setting
\begin{gather}\label{eqn:prp-opt:lambdasg}
	\lambda_\sg(\tau,z):=\sum_{i,\chi} n_{i,\chi}\chi(\sg)\varphi_i(\tau,z)
\end{gather}
in case 
$\lambda=\sum_{i,\chi} n_{i,\chi}
\varphi_i
\otimes 
\chi
$ for some integers $n_{i,\chi}$, for some subset $\{\varphi_i\}\subset S_{k,m}(\#\sG)_\ZZ$.
We then define a (generally non-integer-valued) symmetric bilinear form $(\cdot\,,\cdot)$ on 
$L_{k,m}(\sG)$ by setting
\begin{gather}\label{eqn:prp-opt:lambdalambda'}
	( \lambda,\lambda')
	:=
	\frac{1}{\#\sG}
	\sum_{\sg\in \sG}
	\langle \lambda_{\sg},\lambda'_\sg\rangle
\end{gather}
for $\lambda,\lambda'\in L_{k,m}(\sG)$, where $\langle\cdot\,,\cdot\rangle$ 
denotes the Petersson inner product on $S_{k,m}(\#\sG)$, as defined in (\ref{eqn:prp-jac:Pet}).

The significance of the construction (\ref{eqn:prp-opt:LkmsG}) is that we may naturally identify $\mc{W}^\opt_{k,m}{(\sG)}_0$ with a subset of $L_{k,m}(\sG)$, and thereby equip it with the structure of a lattice. 
Indeed, given $W\in \mc{W}^\opt_{k,m}{(\sG)}_0$, and taking $\{\varphi_i\}$ to be a $\ZZ$-basis for $S_{k,m}(\#\sG)_\ZZ$, we have that
\begin{gather}\label{eqn:prp-opt:phiWsgsum}
	\phi^W_\sg = \sum_i m_i(\sg)\varphi_i
\end{gather}
for each $\sg\in \sG$, for some class functions $\sg\mapsto m_i(\sg)$, since $0$-optimal forms are cuspidal by definition (cf.\ (\ref{eqn:prp-jac:varrhomhath_vanishing}), (\ref{eqn:prp-opt:optimal})).
Any class function on $\sG$ is a linear combination of the irreducible characters of $\sG$, so we have 
\begin{gather}\label{eqn:prp-opt:misg}
m_i(\sg)=\sum_{\chi\in\Irr(\sG)}m_{i,\chi}\chi(\sg),
\end{gather} 
for each $i$, for some scalars $m_{i,\chi}\in\CC$. An elementary argument verifies that $m_{i,\chi}\in\ZZ$ for all $i$ and $\chi$.
Thus, substituting (\ref{eqn:prp-opt:misg}) into (\ref{eqn:prp-opt:phiWsgsum}), and writing $\phi^W$ for the $S_{k,m}(\#\sG)$-valued class function on $\sG$ given by $\sg\mapsto \phi^W_\sg$ (\ref{eqn:prp-opt:phiWsgsum}), we obtain the identification
\begin{gather}\label{eqn:prp-opt:phiW}
	\phi^W = \sum_{i,\chi} m_{i,\chi}\varphi_i\otimes \chi
\end{gather}
of $\phi^W$ as an element of $L_{k,m}(\sG)$, and the association $W\mapsto \phi^W$ defines the promised embedding 
\begin{gather}\label{eqn:prp-opt:mcWsGtoLkmsG}
	\mc{W}^\opt_{k,m}{(\sG)}_0\to L_{k,m}(\sG).
\end{gather}

We henceforth write $\mc{L}^\opt_{k,m}(\sG)$ for the lattice structure on $\mc{W}^\opt_{k,m}{(\sG)}_0$ that we obtain by setting
\begin{gather}\label{eqn:prp-opt:langleWWprangle}
	( W,W') := (\phi^W,\phi^{W'})
\end{gather}
for $W,W'\in \mc{L}^\opt_{k,m}(\sG) = \mc{W}^\opt_{k,m}{(\sG)}_0$, where $(\cdot\,,\cdot)$ on the right-hand side of (\ref{eqn:prp-opt:langleWWprangle}) is the symmetric bilinear form (\ref{eqn:prp-opt:lambdalambda'}) on $L_{k,m}(\sG)$.

Now define $\opc^\opt_{k,m}(\sG)$ to be the minimal positive integer $\opc$ for which a $\opc$-optimal virtual graded $\sG$-module of weight $k$ and index $m$ exists, 
\begin{gather}\label{eqn:prp-opt:opcopt}
	\opc^\opt_{k,m}(\sG) := \min\left\{\opc\in \ZZ^+\left|\, \mc{W}^\opt_{k,m}{(\sG)}_\opc\neq \emptyset\right.\right\}.
\end{gather}
The next result follows from our remarks leading to (\ref{eqn:prp-opt:mcWoptkmsG}), the finiteness of the rank of $S_{k,m}(\#\sG)_\ZZ$, and the fact that the Fourier coefficients of the McKay--Thompson series $\phi^W_\sg$ (\ref{eqn:prp-opt:phiWsg}) are algebraic integers of bounded degree by construction (\ref{eqn:prp-opt:hWsgr}). 
\begin{pro}\label{pro:prp-opt:mcWG}
The sets $\mc{W}^\opt_{k,m}(\sG)$ and  $\mc{L}^\opt_{k,m}(\sG)$ are naturally free abelian groups of finite rank. 
If $\opc\equiv 0\xmod \opc^\opt_{k,m}(\sG)$ then $\mc{W}^\opt_{k,m}{(\sG)}_\opc$ is naturally a $\mc{L}^\opt_{k,m}(\sG)$-torsor. 
If $\opc\not\equiv 0\xmod \opc^\opt_{k,m}(\sG)$ then $\mc{W}^\sG_{k,m}{(\sG)}_\opc$ is empty. 
\end{pro}

We may interpret Proposition \ref{pro:prp-opt:mcWG} as saying that the optimal virtual graded $\sG$-modules of weight $k$ and index $m$ are classified by $\opc^\opt_{k,m}(\sG)$ and the lattice $\mc{L}^\opt_{k,m}(\sG)$. 
For this reason we refer to the determination of $\opc^\opt_{k,m}(\sG)$ and $\mc{L}^\opt_{k,m}(\sG)$ as the {\em classification problem} for optimal virtual graded $\sG$-modules of weight $k$ and index $m$.

To conclude this section we 
mention that there is a counterpart operation $W\mapsto \check{W}$ on modules $W$ as in (\ref{eqn:prp-opt:W}), to the construction (\ref{eqn:prp-jac:checkh}) that translates from Jacobi forms of integer weight to modular forms in Kohnen plus spaces. 
Namely, given a virtual graded $\sG$-module $W$ as in (\ref{eqn:prp-opt:W}) we may consider the virtual graded $\sG$-module $\check{W}=\bigoplus_D \check{W}_D$ determined by setting
\begin{gather}\label{eqn:prp-opt:checkW}
	\check{W}_D := \sum_{\substack{r\xmod 2m\\D\equiv r^2\xmod 4m}} W_{r,\frac{D}{4m}}.
\end{gather}
Then the McKay--Thompson series associated to the action of $\sG$ on $\check{W}$ are none other than the functions $\check{h}^W_\sg$ (cf.\ (\ref{eqn:prp-jac:checkh})), where $h^W_\sg=(h^W_{\sg,r})$ is as in (\ref{eqn:prp-opt:hWsgr}). 
That is to say, we have
\begin{gather}\label{eqn:prp-opt:checkhWg}
	\check{h}^W_\sg(\tau)=\sum_D \tr\left(\sg\left|\check{W}_D\right.\right)q^{-D}
\end{gather}
for $\sg\in \sG$.

\subsection{Our Main Focus}\label{sec:prp-fcs}

In \S\S~\ref{sec:prp-jac}--\ref{sec:prp-opt} we have 
discussed mock Jacobi forms in arbitrary integer weight and positive integer index. 
In this section we specialize to the situation upon which we focus in the remainder of this work, whereby the weight is $2$ and the index is $1$. 
Furthermore, we explain how it is that the constructions of \S~\ref{sec:prp-cln} provide examples of the structures discussed in \S\S~\ref{sec:prp-jac}--\ref{sec:prp-opt}.

Observe that with 
the 
specialization to 
$m=1$ 
in place 
the theta-decomposition (\ref{eqn:prp-jac:thetadecomp}) simplifies to
\begin{gather}\label{eqn:prp-jac:thetadecomp_indexone}
	\phi(\tau,z) = h_0(\tau)\theta_{1,0}(\tau,z) + h_1(\tau)\theta_{1,1}(\tau,z)
\end{gather}
(cf.\ (\ref{eqn:int-mth:thetadecomp_phiW})), and this entails a simplification in the notation for the Fourier coefficients 
of the $h_r$. 
Specifically, the term $C_\phi(D,r)$ in (\ref{eqn:prp-jac:hr}) depends now only on $D$ and the parity of $r$, and since $D\equiv r^2\xmod 4$ the parity of $D$ and $r$ must match.
So we have
$C_\phi(D,r)=C_\phi(D,D)$ for all $D$, and thus there is no loss in
simply 
writing $C_\phi(D)$. 
Putting this together with (\ref{eqn:prp-jac:thetamr}) and (\ref{eqn:prp-jac:thetadecomp_indexone}) we obtain the Fourier series expansion 
\begin{gather}\label{eqn:prp-fcs:phi}
	\phi(\tau,z)=\sum_{n,s\in\ZZ}C_\phi(s^2-4n)q^ny^s
\end{gather} 
for $\phi$ of index $1$ 
(cf.\ (\ref{eqn:prp-cln:msHN}), (\ref{eqn:prp-cln:msHCohN})),
where $C_\phi(D)=0$ for $D>0$ in case $\phi$ is 
holomorphic (cf.\ (\ref{eqn:prp-jac:hr}), (\ref{eqn:prp-jac:varrhomhath_bounded})).
In particular, the coefficient of $q^ny^s$ in (\ref{eqn:prp-fcs:phi}) depends only on the discriminant $D=s^2-4n$ (cf.\ \S~\ref{sec:prp-cln}).
Note that a similar simplification of notation (\ref{eqn:prp-fcs:phi}) can be made in the case of prime index as well, on the strength of the comments that follow (\ref{eqn:prp-jac:checkh}).

Observe also that for $\phi$ a mock Jacobi form of index $1$ the construction $h\mapsto \check{h}$ of (\ref{eqn:prp-jac:checkh}) specializes to 	
\begin{gather}\label{eqn:prp-jac:checkh_indexone}
	\check h(\tau) := h_0(4\tau)+h_1(4\tau)=\sum_D C_\phi(D)q^{-D},
\end{gather}
where 
the $h_r$ are as in (\ref{eqn:prp-jac:thetadecomp_indexone}) and $C_\phi(D)$ is as in (\ref{eqn:prp-fcs:phi}),
and the corresponding construction $W\mapsto \check{W}$ of
(\ref{eqn:prp-opt:checkW})
is given in index $1$ by setting 
$\check{W}_D=
W_{D,\frac{D}{4}}$. 
Thus there is no loss of information in considering $\check{W}=\bigoplus_D \check{W}_D$ in place of $W$, so we do so. 
And again, we may do similarly, when the index $m$ is prime. However, we henceforth drop the accent from $\check{W}$ in order to simplify notation. That is, 
we write simply
\begin{gather}\label{eqn:prp-fcs:W}
	W=\bigoplus_D W_D
\end{gather}
{\color{black} to indicate} the grading of a virtual graded $\sG$-module $W$ as in (\ref{eqn:prp-opt:W}) when $m=1$, where $W_D$ in (\ref{eqn:prp-fcs:W}) is $W_{D,\frac{D}{4}}$ in (\ref{eqn:prp-opt:W}).
With this convention the associated McKay--Thompson series $\phi^W_\sg$ (see (\ref{eqn:prp-opt:phiWsg}--\ref{eqn:prp-opt:hWsgr})) may be defined succinctly by setting
\begin{gather}\label{eqn:prp-fcs:phiWsg}
	\phi^W_\sg(\tau,z):=\sum_{n,s\in\ZZ}\tr\left(\sg\left|W_{s^2-4n}\right.\right)q^ny^s
\end{gather}
(cf.\ (\ref{eqn:prp-opt:trsgV}), (\ref{eqn:prp-fcs:phi})).

Everything we have said so far in this section applies to mock Jacobi forms $\phi$ of index $1$ with arbitrary integer weight $k$ (and with suitable modifications, also to forms with prime index). Now taking the specialization to $k=2$ into account 
we obtain that the relevant skew-holomorphic Jacobi forms belong to $J^\sk_{1,1}(N)$ according to (\ref{eqn:prp-jac:xi_ses}), and in particular have weight $1$. 
In the terminology of \cite{MR4127159}, such skew-holomorphic Jacobi forms are of {\em theta-type}, by force of a result \cite{MR0472707} of Serre--Stark. 
That is, for $\varphi\in J^\sk_{1,1}(N)$ we must have a theta-decomposition 
\begin{gather}\label{eqn:prp-jac:thetadecomp_skewindexone}
	\vf(\tau,z) = \overline{f_0(\tau)} \theta_{1,0}(\tau,z) + \overline{f_1(\tau)}\theta_{1,1}(\tau,z)
\end{gather}
(cf.\ (\ref{eqn:prp-jac:thetadecomp}), (\ref{eqn:prp-jac:thetadecomp_indexone})), where the complex conjugates $f_r$ of the theta-coefficients $\overline{f_r}$ of $\vf$ are linear combinations of the {\em Thetanullwerte} 
\begin{gather}\label{eqn:prp-fcs:Thetanullwerte}
	\theta_{m,r}^0(\tau):=\theta_{m,r}(\tau,0)
\end{gather}
(cf.\ (\ref{eqn:prp-jac:thetamr})).

As we alluded to in \S~\ref{sec:prp-cln}, the 
generating functions $\sHGH_N$ (\ref{eqn:prp-cln:msHN})
and $\sHCE_N$ (\ref{eqn:prp-cln:msHCohN})
are (mock) modular. 
To put this precisely 
we define a theta-type skew-holomorphic Jacobi form $t_d\in J^\sk_{1,1}(d^2)$, for $d$ a positive integer, by setting $t_d:=t_d^0+t_d^1$, where
\begin{gather}\label{eqn:prp-fcs:tdk}
t^s_{d}(\tau,z):=\overline{\theta_{d^2,sd^2}^0(\tau)}\theta_{1,sd}(\tau,z)
\end{gather}
(cf.\ (\ref{eqn:prp-jac:thetadecomp_skewindexone})), 
for $\theta_{m,r}^0$ as in (\ref{eqn:prp-fcs:Thetanullwerte}), and $\theta_{m,r}$ as in (\ref{eqn:prp-jac:thetamr}).
The most important case is $d=1$, on account of the fact that
\begin{gather}\label{eqn:prp-fcs:t1}
	t_1(\tau,z) = \overline{\theta_{1,0}^0(\tau)}\theta_{1,0}(\tau,z) + \overline{\theta_{1,1}^0(\tau)}\theta_{1,1}(\tau,z)
\end{gather}
spans $J^\sk_{1,1}(N)$ for $N$ square-free, and belongs to $J^\sk_{1,1}(N)$ for every $N$.
\footnote{Concrete methods for computing spaces of holomorphic and skew-holomorphic Jacobi forms, including $J^\sk_{1,1}(N)$, may be found in \cite{MR2512363}.}

\begin{pro}\label{pro:prp-fcs:msHNmock}
For $N$ a positive integer the function $\sHGH_N$ of (\ref{eqn:prp-cln:msHN}) belongs to $\JJ_{2,1}(N)$, and the 
shadow of $\sHGH_N$ is a linear combination of the functions $t_{d}$ for which $d$ is a positive integer such that $d^2|N$. 
For $N$ prime the function $\sHCE_N$ of (\ref{eqn:prp-cln:msHCohN}) belongs to $J_{2,1}(N)$. 
\end{pro}
\begin{proof}
The statements about $\sHGH_N$ for $N=1$ 
follow directly from the Corollary in \S~2.2 of \cite{MR0453649} (see also \cite{MR0429750}). 
For general $N$ we may apply the results of \cite{MR1930980}. 
The statement about $\sHCE_N$ 
for $N$ prime follows from the discussion in \S~12 of \cite{MR894322}.
\end{proof}

One consequence of Proposition \ref{pro:prp-fcs:msHNmock} is that 
$\xi(\sHGH_N)$ is a multiple of $t_1$ (\ref{eqn:prp-fcs:t1}) when $N$ is square-free.

Note that the 
definition of the shadow map $\xi$ in (\ref{eqn:prp-jac:xi_ses}) 
requires the choice of an overall constant,  
so in order to be completely concrete about the shadow of $\sHGH_N$ let us fix this choice 
by requiring that 
\begin{gather}\label{eqn:prp-jac:shadowmsH1}
	\xi(\sHGH_1) = \frac12t_1
\end{gather}
(cf.\ (\ref{eqn:prp-fcs:t1})). 
\footnote{Motivation for this particular choice will be more clear in our companion papers \cite{radhur,hurm23}.}
With this normalization (\ref{eqn:prp-jac:shadowmsH1}) in place we 
have 
\begin{gather}\label{eqn:prp-jac:shadowmsHN}
	\xi(\sHGH_N) = t_1
\end{gather}
for $N$ prime,
according to Lemma \ref{lem:cln-ces:ceshcn}.

Except for $\sHGH=\sHGH_1$ (cf.\ (\ref{eqn:prp-cln:msH1})) the $\sHGH_N$ are not optimal (cf.\ (\ref{eqn:prp-opt:optimal}--\ref{eqn:prp-opt:c-optimal})), but we may use them to define our main examples. 
For this we set
\begin{gather}\label{eqn:prp-fcs:msRN}
	\sHR_N
	:=\frac{12}{\phi(N)}\sum_{M|N}\mu\left(\frac NM\right)\frac{M}{\indo(M)}\sHGH_M,
\end{gather}
for $N$ a positive integer,
where $\phi(N)$ denotes the Euler totient function, $\mu(N)$ is the M\"obius function, and $\indo(N)$ is as in 
(\ref{eqn:prp-cln:HN0}). 
\begin{pro}\label{pro:prp-fcs:msRN_1opt}
For any positive integer $N$ the function $\sHR_N$ is a $1$-optimal element of $\JJ_{2,1}(N)$. 
\end{pro}
\begin{proof}
The statement that $\sHR_N$ belongs to $\JJ_{2,1}(N)$ follows from the construction (\ref{eqn:prp-fcs:msRN}), and Proposition \ref{pro:prp-fcs:msHNmock}. 
The statement that
\begin{gather}\label{eqn:prp-fcs:msRN_1opt}
	\sHR_N(\tau,z)=-1+O(q)
\end{gather}
as $\Im(\tau)\to \infty$, 
for any fixed $z$,
also follows from the construction. 
Thus $\sHR_N$ is $1$-optimal (cf.\ (\ref{eqn:prp-opt:c-optimal})) if is optimal (\ref{eqn:prp-opt:optimal}). The optimality of $\sHR_N$ follows from Proposition 5.2 of \cite{MR4291251}.
\end{proof}

\begin{rmk}
To motivate the notation in (\ref{eqn:prp-fcs:msRN}) we mention that the fact that $\sHR_N$ is optimal means that it may be expressed as a Rademacher sum. We refer to \S~2 of \cite{MR4291251} for more detail on this.
\end{rmk}

Note that the Fourier coefficients of $\sHR_N$ are rational numbers that, except when $N=1$, are generally not integers. 
Also, $\JJ_{2,1}(1)$ is spanned by $\sHR_1=12\sHGH$ (cf.\ (\ref{eqn:prp-cln:msH1}) and see Proposition \ref{pro:prp-fcs:JJ21NJ21N}). So for any finite group $\sG$, 
a $\opc$-optimal 
virtual graded $\sG$-module 
$W=\bigoplus_D W_D$ 
of weight $2$ and index $1$ 
(cf.\ (\ref{eqn:prp-fcs:W}--\ref{eqn:prp-fcs:phiWsg}))
satisfies
\begin{gather}\label{eqn:prp-fcs:phiWse}
	\begin{split}
	\phi^W_\se(\tau,z) &= \sum_{n,s\in\ZZ} \dim(W_{s^2-4n})q^ny^s
	\\
	 &
	 = \opc\sHR_1(\tau,z)\\
	 &= 12\opc\sHGH(\tau,z),
	\end{split}
\end{gather}
where $\se$ denotes the identity element of $\sG$. 
In particular, it follows from (\ref{eqn:prp-fcs:phiWse}) that $\phi^W_\se$ vanishes identically if $W\in\mc{L}^\opt_{2,1}(\sG)$ (cf.\ (\ref{eqn:prp-opt:langleWWprangle})). 
Thus the embedding (\ref{eqn:prp-opt:mcWsGtoLkmsG}) refines to a map
\begin{gather}\label{eqn:prp-fcs:mcLsG}
	\mc{L}^\opt_{2,1}(\sG)\to L_{2,1}(\sG)_0
\end{gather}
in the situation at hand, where $L_{2,1}(\sG)_0$ 
(cf.\ (\ref{eqn:prp-opt:LkmsG}))
denotes the kernel of the map $\lambda\mapsto \lambda_\se$ 
(cf.\ (\ref{eqn:prp-opt:lambdasg})).
In other words, 
\begin{gather}\label{eqn:prp-fcs:L21sG0}
	L_{2,1}(\sG)_0:=S_{2,1}(\#\sG)\otimes R(\sG)_0,
\end{gather}
where $R(\sG)_0$ is composed of the virtual modules 
$V\in R(\sG)$ (see (\ref{eqn:prp-opt:RsG})) such that $\tr(\se|V)=0$ (cf.\ (\ref{eqn:prp-opt:trsgV})).

In the sequel \S~\ref{sec:res} we will 
focus on the case that $\sG=\sZNZ$ is a cyclic group of prime order.
Thus, in addition to (\ref{eqn:prp-fcs:phiWse}) we are interested in the optimal level $N$ forms in $\JJ_{2,1}(N)$, for $N$ prime. 
To get a sense for how these forms look we first note the following two results.
\begin{pro}\label{pro:prp-fcs:JJ21NJ21N}
The space $\JJ_{2,1}(1)$ is spanned by $\sHGH$. More generally, $\JJ_{2,1}(N)/ J_{2,1}(N)$ is spanned by the image of $\sHGH$ when $N$ is square-free.
\end{pro}
\begin{proof}
As mentioned in \S~\ref{sec:prp-fcs} (cf.\ (\ref{eqn:prp-fcs:t1})), the methods of \cite{MR2512363} may be used to check that $J^\sk_{1,1}(N)$ is spanned by $t_1$ (see (\ref{eqn:prp-fcs:t1})) when $N$ is square-free. 
Thus the second statement follows from the $k=2$ and $m=1$ case of the exact sequence (\ref{eqn:prp-jac:xi_ses}). 
For the first statement we take $N=1$ in the second statement, and note that $J_{2,1}(1)=\{0\}$ according to Theorem 3.5 in \cite{MR781735}.
\end{proof}
\begin{pro}\label{pro:prp-fcs:JJ21NS21N}
If $N$ is prime then $\JJ_{2,1}(N)/S_{2,1}(N)$ is spanned by the images of $\sHGH$ and $\sHGH_N$.
\end{pro}
\begin{proof}
It follows from the methods of \cite{MR660784} and \cite{MR1247596} (see also \cite{MR819403} and \cite{MR1637086}) that $J_{2,1}(N)$ is isomorphic, 
as a module for the Hecke algebra ${\bf T}_N$ of level $N$,
to the space 
of holomorphic modular forms of weight $2$ for $\Gamma_0(N)$ for $N$ prime. 

When $N$ is prime there is a unique (modular) Eisenstein series of weight $2$ for $\Gamma_0(N)$, and it may be checked that  
it has the same Hecke eigenvalues as 
the Cohen--Eisenstein series $\sHCE_N$. Thus $\JJ_{2,1}(N)/S_{2,1}(N)$ is spanned by the images of $\sHGH$ and $\sHCE_N$ according to Proposition \ref{pro:prp-fcs:JJ21NJ21N}, and the claimed result then follows from Lemma 
\ref{lem:cln-ces:ceshcn}. 
\end{proof}

Suppose now that $W\in \mc{W}^\opt_{k,m}(\sG)_\opc$ is a $\opc$-optimal module for $\sG=\ZZ/N\ZZ$ for $N$ prime. Then, on the strength of 
Proposition \ref{pro:prp-fcs:JJ21NS21N},
for a non-identity element $\sg\in \sG$ we must have
\begin{gather}
\begin{split}\label{eqn:prp-fcs:phiWsg-spec}
	\phi^W_\sg(\tau,z) 
	&= \sum_{n,s\in\ZZ} \tr(\sg|W_{s^2-4n})q^ny^s
	\\
	&
	=\opc\sHR_{N}(\tau,z)+\varphi_N(\tau,z)\\
	&=\frac{N}{N^2-1}12\opc\sHGH_N(\tau,z) - \frac{1}{N-1}12\opc\sHGH(\tau,z)+\varphi_N(\tau,z),
\end{split}
\end{gather}
(cf.\ (\ref{eqn:prp-fcs:msRN})) for some cuspidal form $\varphi_N\in S_{2,1}(N)$. 
Thus the problem of computing $\opc^\opt_{2,1}(\sG)$ (see (\ref{eqn:prp-opt:opcopt}))
boils down to the problem of finding the minimal positive integer $\opc$ such that 
\begin{gather}\label{eqn:prp-fcs:opcopt21}
\frac{N}{N^2-1}12\opc\sHGH_N(\tau,z) - \frac{1}{N-1}12\opc\sHGH(\tau,z)+\varphi_N(\tau,z)
\end{gather}
belongs to $\JJ_{2,1}(N)_\ZZ$ (see (\ref{eqn:prp-jac:JJkmGamma0JNZZ})) for some $\varphi_N\in S_{2,1}(N)$.

With the preceding as motivation we conclude this section by singling out the subspace 
\begin{gather}\label{eqn:prp-fcs:JEiskm}
	\JJ^\Eis_{2,1}(N) := \Span\left\{ \left.\sHGH_M\,\right| \text{$M$ divides $N$} \right\}
\end{gather}
of $\JJ_{2,1}(N)$, 
which we refer to as the space of {\em Eisenstein series} in $\JJ_{2,1}(N)$. 
According to Proposition \ref{pro:prp-fcs:JJ21NS21N}
we have that $\JJ^\Eis_{2,1}(N)$ is a complement 
to the subspace 
of cuspidal Jacobi forms in $\JJ_{2,1}(N)$, so that
\begin{gather}\label{eqn:prp-fcs:Eiscsp_decomp}
	\JJ_{2,1}(N) = \JJ^\Eis_{2,1}(N)\oplus S_{2,1}(N),
\end{gather}
at least when $N$ is prime.

\section{Results}\label{sec:res}

In this section we prove our main results. 
The first of these is a 
solution to the classification problem formulated in \S~\ref{sec:prp-opt} 
(cf.\ Proposition \ref{pro:prp-opt:mcWG})
in the setting described in \S~\ref{sec:prp-fcs}, for the cyclic groups of prime order. 
We establish this classification in 
\S~\ref{sec:res-opt}.
Then, in \S~\ref{sec:res-art}, we formulate consequences of this classification for the arithmetic of imaginary quadratic twists of modular abelian varieties.

\subsection{Modules}\label{sec:res-opt}

Here we solve the optimal module classification problem formulated in \S~\ref{sec:prp-opt} 
(cf.\ Proposition \ref{pro:prp-opt:mcWG}), 
in the setting described in \S~\ref{sec:prp-fcs}, for the cyclic groups of prime order. 
That is, we determine 
$\mc{L}^\opt_{2,1}(\sG)$ (cf.\ (\ref{eqn:prp-opt:langleWWprangle}), (\ref{eqn:prp-fcs:mcLsG})) 
and 
$\opc^\opt_{2,1}(\sG)$ (cf.\ (\ref{eqn:prp-opt:opcopt}), (\ref{eqn:prp-fcs:opcopt21})) 
precisely,
where here---and from here on---$\sG$ denotes $\ZZ/N\ZZ$ and $N$ denotes a prime. 
To formulate the result we define
\begin{gather}\label{eqn:res-opt:opceop}
	\opc^\eop(N):=
	\begin{cases}
	1&\text{if $N=2$,}\\
	2&\text{if $N=3$,}\\
	\frac{N^2-1}{24}&\text{if $N\equiv 1\xmod 4$,}\\
	\frac{N^2-1}{12}&\text{if $N\equiv 3\xmod 4$ and $N>3$,}
	\end{cases}
\end{gather}
we write $J_0(N)$ for the Jacobian of the modular curve $X_0(N)$ defined by $\Gamma_0(N)$ (\ref{eqn:prp-cln:Gamma0N}), and write $J_0(N)(\QQ)_\tor$ for the torsion subgroup of the group $J_0(N)(\QQ)$ of $\QQ$-rational points on $J_0(N)$. 
We also recall the embedding (\ref{eqn:prp-fcs:mcLsG}), and use it to regard $\mc{L}^\opt_{2,1}(\sG)$ 
as a subset of $L_{2,1}(\sG)_0$ (see (\ref{eqn:prp-fcs:L21sG0})).

\begin{thm}\label{thm:res-opt:tor}
For $N$ prime and $\sG=\sZNZ$ we have 
$\mc{L}^\opt_{2,1}(\sG)=
L_{2,1}(\sG)_0$ 
and
\begin{gather}\label{eqn:res-opt:tor}
\opc^\opt_{2,1}(\sG)=\frac{\opc^\eop(N)}{\#J_0(N)(\QQ)_\tor}.
\end{gather}
\end{thm}

It may not be 
clear from (\ref{eqn:res-opt:tor}) that $\opc^\opt_{2,1}(\sG)$ is a rational integer. For $\alpha$ a positive rational number let $\num(\alpha)$ denote the numerator of $\alpha$, when expressed in reduced form, and define
\begin{gather}\label{eqn:res-opt:nHNnCN}
	\nH_N:=\num\left(\frac{N+1}{6}\right),\quad
	\nC_N:=\num\left(\frac{N-1}{12}\right)
\end{gather}
(cf.\ (\ref{eqn:prp-cln:dHNdCN})). 
Then we have 
$\opc^\eop(N)=\nH_N\nC_N$. 
Also, we have 
$\#J_0(N)(\QQ)_\tor=\nC_N$ 
for prime $N$, according to a result of Mazur \cite{MR488287}. 
(Note that $\nC_N$ is denoted $n_N$ in op.\ cit.)
So 
the identity (\ref{eqn:res-opt:tor}) may be reformulated as the statement that
\begin{gather}\label{eqn:res-opt:opcoptnHN}
\opc^\opt_{2,1}(\sG)=\nH_N=\num\left(\frac{N+1}6\right),
\end{gather}
for $\sG=\ZZ/N\ZZ$.

Note also that $S_{2,1}(N)$ is naturally isomorphic to the space $S_2(N)$ of cuspidal modular forms of weight $2$ for $\Gamma_0(N)$, when $N$ is prime, according to \cite{MR1637086} (and this can also be seen using \cite{MR660784}). 
Thus we may express our computation of $\mc{L}^\opt_{2,1}(\sG)$, for $\sG=\ZZ/N\ZZ$ for $N$ prime, by writing
\begin{gather}\label{eqn:res-opt:mcLoptS2NRG0}
	\mc{L}^\opt_{2,1}(\sG) = L_{2,1}(\sG)_0 = S_2(N)\otimes R(\sG)_0.
\end{gather}

As preparation for the proof of Theorem \ref{thm:res-opt:tor} we first consider the analogous situation wherein the graded trace functions (\ref{eqn:prp-opt:phiWsg}--\ref{eqn:prp-opt:hWsgr}) arising are restricted to be Eisenstein series, as defined in (\ref{eqn:prp-fcs:JEiskm}) (cf.\ (\ref{eqn:prp-fcs:Eiscsp_decomp})).
In fact this situation is very constrained, because according to 
Propositions 
\ref{pro:prp-fcs:msRN_1opt}
and 
\ref{pro:prp-fcs:JJ21NS21N} 
the unique-up-to-scale optimal Eisenstein series for $\GammaOJ(N)$ is $\sHR_N$ (\ref{eqn:prp-fcs:msRN}).
So the task at hand is to determine the minimal positive integer $\opc^\eop(\sG)$ 
for which there exists a virtual graded $\sG$-module $W=\bigoplus_D W_D$ (\ref{eqn:prp-fcs:W}) 
of weight $2$ and index $1$ 
such that 
\begin{gather}\label{eqn:res-opt:opceopG}
\phi^W_\sg=\opc^\eop(\sG)\sHR_{o(\sg)}
\end{gather} 
(cf.\ (\ref{eqn:prp-fcs:phiWse}--\ref{eqn:prp-fcs:phiWsg-spec})).
Given our choice of notation in (\ref{eqn:res-opt:opceop}),
the reader may expect, or hope, that
$\opc^\eop(\sG)=\opc^\eop(N)$. We confirm this identity next.

\begin{pro}\label{pro:res-opt:torEis}
For $N$ prime and $\sG=\ZZ/N\ZZ$ we have $\opc^\eop(\sG)=\opc^\eop(N)$. 
\end{pro}
\begin{proof}
We first show that $\opc^\eop(\sG)$ divides $\opc^\eop(N)$, by verifying that
there exists a virtual graded $\sZNZ$-module $W$ such that 
$\phi^W_\sg=\opc^\eop(N)\sHR_{o(\sg)}$ 
for $\sg\in \sZNZ$. For this we require to check that $\opc^\eop(N)\sHR_1$ and $\opc^\eop(N)\sHR_N$ have integer coefficients, and satisfy the congruence 
\begin{gather}\label{eqn:res-opt:mstarNR1equivmstarNRNmodN}
\opc^\eop(N)\sHR_1
\equiv 
\opc^\eop(N)\sHR_N
\xmod N.
\end{gather} 

To carry out this check we take $\opc=\opc^\eop(N)$ in the last identity in each of (\ref{eqn:prp-fcs:phiWse}--\ref{eqn:prp-fcs:phiWsg-spec}), 
and apply Lemma \ref{lem:cln-ces:ceshcn} 
in order to rewrite the $\sHGH$ appearing in terms of $\sHGH_N$ and $\sHCE_N$. This yields
\begin{gather}
	\opc^\eop(N)\sHR_1
	=\label{eqn:res-opt:mstarNmsR1}
	(N+1)\dH_N\nC_N\sHGH_N+(N-1)\dC_N\nH_N\sHCE_N,
	\\
	\opc^\eop(N)\sHR_N 
	=\label{eqn:res-opt:mstarNmsRN}
	\dH_N\nC_N\sHGH_N-\dC_N\nH_N\sHCE_N,
	\\
	\opc^\eop(N)\sHR_1
	-
	\opc^\eop(N)\sHR_N 
	=\label{eqn:res-opt:mstarNmsR1-mstarNmsRN}
	N\dH_N\nC_N\sHGH_N+N\dC_N\nH_N\sHCE_N,
\end{gather}
where $\dH_N$ and $\dC_N$ are as in (\ref{eqn:prp-cln:dHNdCN}), and $\nH_N$ and $\nC_N$ are as in (\ref{eqn:res-opt:nHNnCN}). 
Then comparison with Lemma \ref{lem:cln-ces:cNHNDdNHCND} confirms that all the coefficients in (\ref{eqn:res-opt:mstarNmsR1}--\ref{eqn:res-opt:mstarNmsRN}) are integers, and all the coefficients in (\ref{eqn:res-opt:mstarNmsR1-mstarNmsRN}) are integers divisible by $N$. 

We have shown that $\opc^\eop(\sG)$ divides $\opc^\eop(N)$. Thus we have $\opc^\eop(\sG)=\opc^\eop(N)$ for $N=2$ and for $N=5$, because in these cases $\opc^\eop(N)=1$. More generally, we may verify that $\opc^\eop(\sG)=\opc^\eop(N)$ by finding a pair of coprime Fourier coefficients of $\opc^\eop(N)\sHR_N$. For the remainder of this proof, let us write $C^\eop_N(D)$ for 
the coefficient of $q^ny^s$ in $\opc^\eop(N)\sHR_N$, when $D=s^2-4n$. Then for $D$ negative and fundamental we have
\begin{gather}\label{eqn:res-opt:torEis-cff}
	C^\eop_N(D)=	
	\dH_N \nC_N \left(1+\left(\frac{D}{N}\right)\right) \HGH(D) - \frac12 \dC_N\nH_N \left(1-\left(\frac{D}{N}\right)\right) \HGH(D),
\end{gather}
according to (\ref{eqn:res-opt:mstarNmsRN}) and Lemma \ref{lem:HCohNDHNDHD}. Thus we obtain $\opc^\eop(\sG)=\opc^\eop(N)$ for $N=3$, for example, by using (\ref{eqn:res-opt:torEis-cff}) to compute that $C^\eop_3(-3)=-1$.
Our main tool for handling more general $N$ will be the application of Theorem A from \cite{MR3404031}. We will also 
freely apply the fact that 
$\dH_N$ and $\nH_N$ 
are coprime by construction
(cf.\ (\ref{eqn:prp-cln:dHNdCN}), (\ref{eqn:res-opt:nHNnCN})), 
and similarly for 
$\dC_N$ and $\nC_N$, and $\nH_N$ and $\nC_N$,
and even 
$\dH_N$ and $\dC_N$ 
when $N>3$.

Suppose now that $N>5$ is prime, and $N\equiv 1\xmod 4$. Suppose also that 
$D_{-1}$ is a negative fundamental discriminant such that $(\frac{D_{-1}}N)=-1$.
Then according to (\ref{eqn:res-opt:torEis-cff}) 
we have 
\begin{gather}\label{eqn:res-opt:torEis-CND-1}
	C^\eop_N(D_{-1})=-\dC_N\nH_N\HGH(D_{-1}).
\end{gather}
We also have $C^\eop_N(-4)=\dH_N\nC_N$. 
It follows then, from the coprimalities amongst the $\dH_N$, $\dC_N$, $\nH_N$ and $\nC_N$, that if $p$ is a common prime divisor of all the coefficients of $\opc^\eop(N)\sHR_N$ then $p$ divides $\HGH(D_{-1})$ for all negative fundamental $D_{-1}$ such that $(\frac{D_{-1}}N)=-1$. 
So suppose that $p$ is a common prime divisor of $\HGH(D_{-1})$, for all negative fundamental $D_{-1}$ such that $(\frac{D_{-1}}N)=-1$. Then $p\neq 2$, because 
$\HGH(D_{-1})$ is odd by genus theory (cf.\ e.g.\ \cite{MR963648}) if we take $D_{-1}=-q$, for $q$ a prime such that $q\equiv 3\xmod 4$ and $(\frac{-q}N)=-1$. 
But $p$ is also not odd, for if so we get a contradiction by applying Theorem A of \cite{MR3404031}, with $S=S_-=\{N\}$ and $\ell=p$, and by taking $D_{-1}$ to be the discriminant of the field $L$ that that theorem produces.
So $\opc^\eop(\sG)=\opc^\eop(N)$ when $N\equiv 1\xmod 4$.

It remains to manage the case that $N>5$ satisfies $N\equiv 3\xmod 4$. For this we consider the negative fundamental $D_0$ that are divisible by $N$. For such $D_0$ we have
\begin{gather}\label{eqn:res-opt:torEis-CND0}
	C^\eop_N(D_0)
	=\left(\dH_N\nC_N-\frac{1}2\dC_N\nH_N\right)\HGH(D_0)
\end{gather}
by (\ref{eqn:res-opt:torEis-cff}), and we also have $C^\eop_N(-4)=-\frac{1}2\dC_N\nH_N$. 
Note that $\frac12 \dC_N$ is an integer when $N\equiv 3\xmod 4$ (cf.\ (\ref{eqn:prp-cln:dHNdCN})). Thus, similar to the above, we conclude that a common prime divisor $p$ of the coefficients of $\opc^\eop(N)\sHR_N$ divides $\HGH(D_0)$, for all negative fundamental $D_0$ that are divisible by $N$. 
Now $p$ is not even because $\HGH(D_0)$ is odd, by genus theory, for $D_0=-N$, and we rule out the possibility that $p$ is odd by again applying 
Theorem A of \cite{MR3404031}, but now with $S=S_0=\{N\}$. This completes the proof.
\end{proof}

\begin{rmk}
The full force of \cite{MR3404031} is not required to verify that $\opc^\eop(\sG)=\opc^\eop(N)$ for $N\equiv 5\xmod 12$ or $N\equiv 7\xmod 12$, since in these cases it suffices to compare $C^\eop_N(-3)$ and $C^\eop_N(-4)$. However, this observation does not seem to shorten the proof of Proposition \ref{pro:res-opt:torEis}.
\end{rmk}

We will use the next result 
to bridge the gap between Proposition \ref{pro:res-opt:torEis} and Theorem \ref{thm:res-opt:tor}.
The argument we give utilizes quaternion algebras, and is similar in essence to the proof of Theorem 2.1 in \cite{MR3762695}.
\begin{lem}\label{lem:res-opt:cuspformcong}
Let $N$ be prime. 
Then 
there exists a cuspidal Jacobi form 
$\varphi_N\in S_{2,1}(N)$
that has rational integer Fourier coefficients and satisfies 
\begin{gather}\label{eqn:res-opt:cuspformcong}
\dC_N\left(\sHCE_N-\frac{N-1}{24}\right)\equiv \varphi_N\xmod{\nC_N}.
\end{gather}
\end{lem}
\begin{proof}
Recall that $\frac{N-1}{24}$ is the constant term of $\sHCE_N$ (cf.\ (\ref{eqn:prp-cln:HCohN}--\ref{eqn:prp-cln:msHCohN})). 
So the left-hand side of (\ref{eqn:res-opt:cuspformcong}) has integer coefficients according to Lemma \ref{lem:cln-ces:cNHNDdNHCND}. If $N\in \{2,3,5,7,13\}$ then $S_{2,1}(N)=\{0\}$, but $\nC_N=1$ in these cases (\ref{eqn:res-opt:nHNnCN}) so $\varphi_N=0$ satisfies the required congruence. 
So let us henceforth assume that $N=11$ or $N\geq 17$. Then the genus of $X_0(N)$ is greater than $0$, and $S_{2,1}(N)\neq\{0\}$. 

To proceed we follow 
\cite{MR1896237} in letting $\mathcal{X}$ denote the free $\ZZ$-module generated by the singular points of the geometric fibre of $X_0(N)$ in characteristic $N$. We denote these singular points $\{e_i\}_{i\in I}$, where $I$ is some index set with size 
exactly $1$ more than the genus of $X_0(N)$. 
The $e_i$ are in natural correspondence with the isomorphism classes of supersingular elliptic curves in characteristic $N$. 

Next let $\mathcal{B}$ be a quaternion algebra over $\QQ$ ramified only at $N$ and $\infty$. For $i\in  I$ let $E_i$ be a supersingular elliptic curve in the class corresponding to $e_i$, and set $R_i:=\End(E_i)$. 
Then $R_i$ may be identified with a maximal order in $\mathcal{B}$ in such a way that the norm $\Nm(b)$ of $b\in R_i$ is the degree of $b$ when regarded as an isogeny $E_i\to E_i$ (see \S\S~1--2 of \cite{MR894322}). Set $S_i:=\ZZ+2R_i\subset\mathcal{B}$ and let $S_i^0$ be the set of elements of $S_i$ with trace zero. Then from \S~12 of \cite{MR894322} we obtain that the theta series
\begin{gather}\label{eqn:varthetai}
	\vartheta_i(\tau):=\sum_{b\in S_i^0}q^{\Nm(b)}
\end{gather}
belongs to the Kohnen plus space 
$M_{\frac32}^+(4N)$ 
of modular forms of weight $\frac32$ for $\Gamma_0(4N)$. 
Thus we may consider the $\ZZ$-linear map $\vartheta:\mathcal{X}\to M_{\frac32}^+(4N)$ defined by setting 
$\vartheta(e_i):=\frac12\vartheta_i$.
Note that the image of $\vartheta$ is composed of modular forms whose Fourier coefficients are rational integers, except possibly for their constant terms, which generally lie in $\frac12\ZZ$.

For $i\in I$ define $w_i$ to be half the number of invertible elements of $R_i$. Then we have
\begin{gather}\label{eqn:res-opt:prodwisumwi}
\prod_{i\in I}w_i=\dC_N=\den\left(\frac{N-1}{12}\right),\quad
\sum_{i\in I}\frac1{w_i}=\frac{\nC_N}{\dC_N}=\frac{N-1}{12}
\end{gather} 
(cf.\ (\ref{eqn:prp-cln:dHNdCN}), (\ref{eqn:res-opt:nHNnCN})),
according to \S~1 of \cite{MR894322} (or \S~3 of \cite{MR1896237}), 
and it follows that the expression
\begin{gather}\label{eqn:res-opt:xE}
x_{\rm E}:=\sum_i \frac{\dC_N}{w_i}e_i 
\end{gather} 
defines an element of $\mc{X}$. 
In particular then the Fourier coefficients of $\vartheta(x_{\rm E})$ are integers, except possibly for the constant term. 
We compute
\begin{gather}\label{eqn:res-opt:varthetaxE}
\vartheta(x_E) = \frac{1}{2}\nC_N + O(q)
\end{gather} 
by applying (\ref{eqn:prp-cln:dHNdCN}) and (\ref{eqn:res-opt:nHNnCN}), 
and the second part of (\ref{eqn:res-opt:prodwisumwi}).  

It follows from the methods of \cite{MR660784} that $M_{\frac32}^+(4N)$ is isomorphic to $M_2(N)$ as a Hecke algebra module, 
so in particular, since $N$ is prime, a modular form in either space is cuspidal as soon as it vanishes at the infinite cusp
(cf.\ the proof of Proposition \ref{pro:prp-fcs:JJ21NS21N}). 
Thus, for any $i\in I$ we have by (\ref{eqn:res-opt:varthetaxE}) 
that the function 
$\vartheta(x_{\rm E})-\frac{1}2\nC_N\vartheta_{i}$ 
(cf.\ (\ref{eqn:varthetai}))
is a cusp form with integer Fourier coefficients that is congruent to 
$\vartheta(x_{\rm E})-\frac{1}2\nC_N$ 
modulo $\nC_N$. 
So fix a choice of $i\in I$, and let $C_{\vf_N}(n)$ denote the coefficient of $q^n$ in $\vartheta(x_{\rm E})-\frac{1}2\nC_N\vartheta_{i}$. Then 
\begin{gather}\label{eqn:res-opt:cuspformcong-msG}
\varphi_N(\tau,z):=\sum_{n,s}C_{\vf_N}(s^2-4n)q^ny^s
\end{gather}
is the cuspidal Jacobi form we seek. 
\end{proof}

We are almost ready to 
prove Theorem \ref{thm:res-opt:tor}.
As a final act of preparation 
we 
choose
an inverse $\tilde{N}$ for $N$ modulo $\nC_N$ (\ref{eqn:res-opt:nHNnCN}), and
define
\begin{gather}\label{eqn:res-opt:msHZNZsg}
	\ms{H}^\sZNZ_\sg:=
	\begin{cases}
		\nH_N\sHR_1&\text{ for $o(\sg)=1$,}\\
		\nH_N\sHR_N-\nH_N\frac{N\tilde{N}}{\nC_N}\varphi_N&\text{ for $o(\sg)=N$},
	\end{cases}
\end{gather}
for $\sg\in \sZNZ$, where 
$\varphi_N$ 
is as in Lemma \ref{lem:res-opt:cuspformcong}.

\begin{proof}[Proof of Theorem \ref{thm:res-opt:tor}.]
We first show that $\opc^\opt_{2,1}(\sG)$ divides $\nH_N$, 
by verifying that there exists a $\sZNZ$-module 
$W$ as in (\ref{eqn:prp-fcs:W}) such that $\phi^W_\sg=\ms{H}^{\sZNZ}_\sg$ for $\sg\in \sZNZ$. 
We then show that $\opc^\opt_{2,1}(\sG)=\nH_N$ by a method similar to that which we used for Proposition \ref{pro:res-opt:torEis}, to verify that $\opc^\eop(\sG)=\opc^\eop(N)$. 
We conclude the proof by showing that the embedding (\ref{eqn:prp-fcs:mcLsG}) is surjective.

To ease the exposition let us write $\ms{H}^\sZNZ_{o(\sg)}$ for $\ms{H}^\sZNZ_\sg$. Then to verify that the $\ms{H}^{\sZNZ}_\sg$ are the graded traces arising from some virtual $\sZNZ$-module we must show that $\ms{H}^\sZNZ_1$ and $\ms{H}^\sZNZ_N$ have integer Fourier coefficients, and satisfy the congruence $\ms{H}^\sZNZ_1\equiv \ms{H}^\sZNZ_N\xmod N$. 
For the integrality we note that the forms
\begin{gather}
	-\nH_N+\dH_N\left(\sHGH_N+\frac{N+1}{12}\right),\label{eqn:res-opt:torexp1}
	\\
	\nH_N\frac{\dC_N}{\nC_N}\left(\sHCE_N-\frac{N-1}{24}\right)-\nH_N\frac{N\tilde{N}}{\nC_N}
	\label{eqn:res-opt:torexp2}
	\varphi_N,
\end{gather}
both have integer Fourier coefficients, according to Lemma \ref{lem:cln-ces:cNHNDdNHCND} for (\ref{eqn:res-opt:torexp1}), 
since $-\frac{N+1}{12}$ is the constant term of $\sHGH_N$ (cf.\ (\ref{eqn:prp-cln:HN0})), and by Lemma \ref{lem:res-opt:cuspformcong} for (\ref{eqn:res-opt:torexp2}). 
Now 
$\ms{H}^\sZNZ_N$ is just the sum of the forms in (\ref{eqn:res-opt:torexp1}--\ref{eqn:res-opt:torexp2}), and 
$\ms{H}^\sZNZ_1=\nH_N\sHR_1$ has integer coefficients because $\sHR_1=12\sHGH$ has integer coefficients by construction. 
For the congruence modulo $N$ it suffices to check that $\nC_N\ms{H}^\sZNZ_1\equiv \nC_N\ms{H}^\sZNZ_N\xmod N$ since $N$ is coprime to $\nC_N$ (\ref{eqn:res-opt:nHNnCN}). This congruence follows, in turn, from 
(\ref{eqn:res-opt:msHZNZsg}) {\color{black} together with} the congruence (\ref{eqn:res-opt:mstarNR1equivmstarNRNmodN}) that we proved for Proposition \ref{pro:res-opt:torEis}.

The forms $\ms{H}^\sZNZ_\sg$ defined in (\ref{eqn:res-opt:msHZNZsg}) are optimal by construction, 
so we have shown that $\opc^\opt_{2,1}(\sG)$ divides $\nH_N$. Similar to the situation in Proposition \ref{pro:res-opt:torEis}, we can verify that $\opc^\opt_{2,1}(\sG)=\nH_N$, for any given $N$, by demonstrating that $\ms{H}^\sZNZ_N$ has a pair of coprime coefficients. 
Thus $\opc^\opt_{2,1}(\sG)=\nH_N$ for $N=2$ and $N=5$, because $\nH_N=1$ in these cases according to (\ref{eqn:res-opt:nHNnCN}). 

To go further let us write $H^\sZNZ_N(D)$ for the coefficient of $q^ny^s$ in $\ms{H}^\sZNZ_N$, when 
$D=s^2-4n$. Then we have $H^\sZNZ_N(0)=-\nH_N$ by construction, and we obtain that 
\begin{gather}\label{eqn:res-opt:HsNZNND}
H^\sZNZ_N(D)\equiv \dH_N\HGH_N(D)\xmod \nH_N
\end{gather} 
for all negative $D$,
by
applying Lemma \ref{lem:res-opt:cuspformcong} and the fact that $\ms{H}^\sZNZ_N$ is the sum of the forms in (\ref{eqn:res-opt:torexp1}--\ref{eqn:res-opt:torexp2}). 
So a common divisor of the coefficients of $\ms{H}^\sZNZ_N$ divides the numerator of $\HGH_N(D)$ for all negative $D$, because
$\dH_N$ and $\nH_N$ are coprime by construction (\ref{eqn:prp-cln:dHNdCN}), (\ref{eqn:res-opt:nHNnCN}). 

We have 
$\HGH_N(D)=(1+(\frac{D}N))\HGH(D)$ 
by Lemma \ref{lem:HCohNDHNDHD}
if $D$ is negative and fundamental. 
So $\opc^\opt_{2,1}(\sG)=\nH_N$ when $N\equiv 1\xmod 4$, because then $\HGH_N(-4)=(1+1)\frac12=1$. To see that $\opc^\opt_{2,1}(\sG)=\nH_N$ when $N\equiv 3\xmod 4$ we proceed as in the last paragraph of the proof of Proposition \ref{pro:res-opt:torEis}. The analogue of (\ref{eqn:res-opt:torEis-CND0}) is now $H^\sZNZ_N(D_0)=\dH_N\HGH(D_0)\xmod \nH_N$, for $D_0$ negative, fundamental and divisible by $N$, and we compare this to $H^\sZNZ_N(0)=-\nH_N$. We rule out the possibility that a common divisor 
is even by noting that $\HGH(-N)$ is odd, and we rule out the possibility of an odd common divisor by applying Theorem A of \cite{MR3404031} with $S=S_0=\{N\}$. This completes the verification that $\opc^\opt_{2,1}(\sG)=\nH_N$ (cf.\ (\ref{eqn:res-opt:opcoptnHN})). 

It remains to determine 
$\mc{L}^\opt_{2,1}(\sG)$. 
For this it suffices to show that the embedding (\ref{eqn:prp-fcs:mcLsG}) is surjective, so let $\lambda\in L_{2,1}(\sG)_0$. 
Then we may write $\lambda=\sum_{i,\chi}n_{i,\chi}\varphi_i\otimes \chi$ for some subset $\{\varphi_i\}\subset S_{2,1}(N)_\ZZ$, where the $n_{i,\chi}$ are integers such that
\begin{gather}\label{eqn:res-opt:nichi}
	\sum_{\chi\in\Irr(\sG)}n_{i,\chi}\chi(\se) =0
\end{gather}
for each $i$. Now let $W=\bigoplus_DW_D$ be the virtual graded $\sG$-module determined by setting
\begin{gather}\label{eqn:res-opt:WD}
	W_D:=\sum_{i,\chi}n_{i,\chi}C_{\varphi_i}(D)\chi
\end{gather}
for each $D$. Then $\phi^W_\sg$ is $0$-optimal of weight $2$ and index $1$ in case $o(\sg)=N$ since the $\varphi_i$ belong to $S_{2,1}(\#\sG)$, and also in case $o(\sg)=1$ since $\phi^W_\se$ vanishes identically on account of (\ref{eqn:res-opt:nichi}).
Thus $\phi^W$ belongs to $L_{2,1}(\sG)_0$. We have $\phi^W_\sg=\lambda_\sg$ for all $\sg\in \sG$ by our construction (\ref{eqn:res-opt:WD}) of $W$, so the map $W\mapsto \phi^W$ from $\mc{L}_{2,1}^\opt(\sG)$ to $L_{2,1}(\sG)_0$ is surjective, as we required to show. 
\end{proof}

We record the following consequence of Theorem \ref{thm:res-opt:tor}.
\begin{cor}\label{cor:res-opt:tor}
For $N$ prime and $\sG=\sZNZ$
the rank of  $\mc{L}^\opt_{2,1}(\sG)$ is $(N-1)\dim J_0(N)$.
\end{cor}

\begin{proof}
It follows from the main result of \cite{MR1981614} that the space $S_{2,1}(N)$ admits a basis composed of forms with rational integer coefficients, so
from Theorem \ref{thm:res-opt:tor} we conclude that the rank of $\mc{L}^\opt_{2,1}(\sG)$ is $(N-1)\dim S_{2,1}(N)$.
To obtain the desired result we 
note that $S_{2,1}(N)$ is isomorphic to $S_2(N)$ as a module for the Hecke algebra ${\bf T}_N$ according to 
 \cite{MR1247596} 
(cf.\ Proposition \ref{pro:prp-fcs:JJ21NS21N}).
In particular, these spaces have the same dimension, and the submodules of forms with integer coefficients have the same rank as $\ZZ$-modules, and these dimensions and ranks all coincide. The dimension of $S_2(N)$ is the same as that of $J_0(N)$ so we have $(N-1)\dim J_0(N)$ for the rank of $\mc{L}^\opt_{2,1}(\sG)$, as required.   
\end{proof}

\subsection{Arithmetic}\label{sec:res-art}

In this section we derive a consequence of the classification
of optimal modules for cyclic groups of prime order, Theorem \ref{thm:res-opt:tor}, for the arithmetic geometry of imaginary quadratic twists of modular abelian varieties. 

Recall that $J_0(N)$ denotes the Jacobian of the modular curve defined by $\Gamma_0(N)$, regarded as an abelian variety defined over $\QQ$ (cf.\ \S~\ref{sec:res-opt}). To formulate the main result, Theorem \ref{thm:res-art:N=arb}, we follow \cite{MR482230} in defining an {\em optimal quotient} of $J_0(N)$ to be an abelian variety $A$, also defined over $\QQ$, which admits a surjective map $J_0(N)\to A$ with connected kernel. For $A$ an abelian variety and $D<0$ we define the {\em $D$-twist} of $A$, to be denoted $A\otimes D$, by setting 
\begin{gather}\label{eqn:res-art:AotimesD}
A\otimes D:=A_F, 
\end{gather}
where $A_F$ is as in Definition 5.1 of \cite{MR2331769}, for $F=\QQ(\sqrt{D})$ (and $k=\QQ$).

\begin{thm}\label{thm:res-art:N=arb}
Let $N$ be a prime, let $D$ be a negative fundamental discriminant such that $(\frac{D}{N})=-1$, 
and let $p$ be a divisor of 
$\# J_0(N)(\QQ)_\tor$.
For such $N$, $D$ and $p$, if
\begin{gather}
\HGH(D)\not\equiv 0 \xmod 
p,
\end{gather}
then there exists an optimal quotient $A$ of $J_0(N)$ such that $A\otimes D$ has only finitely many rational points.
\end{thm}
\begin{proof}
For the course of this proof 
we let $\varphi_N\in S_{2,1}(N)$ be a cuspidal Jacobi  form as in (\ref{eqn:res-opt:torexp2}), and write 
$C_N(D)$ for the coefficient of $q^ny^s$ in 
the Fourier expansion 
(cf.\ (\ref{eqn:prp-fcs:phi})) 
of the form that appears in (\ref{eqn:res-opt:torexp2}) when $D=s^2-4n$. 
Then, as mentioned in the proof of Theorem \ref{thm:res-opt:tor}, the $C_N(D)$ are all rational integers, according to Lemma \ref{lem:res-opt:cuspformcong}. 
Furthermore, we have 
\begin{gather}\label{eqn:res-art:N=arb-CNDmodnN}
	C_N(D)\equiv \nH_N(\dC_N\HGH(D)-C_{\varphi_N}(D))\xmod \nC_N,
\end{gather}
where $C_{\varphi_N}(D)$
is as in (\ref{eqn:res-opt:cuspformcong-msG}), because $\HCE_N(D)=\HGH(D)$ under our hypothesis on $D$, according to Lemma \ref{lem:HCohNDHNDHD}.

Suppose, as in the statement of the theorem, that $p$ is a divisor of $\#J_0(N)(\QQ)_\tor$
that does not divide $\HGH(D)$. As mentioned after (\ref{eqn:res-opt:nHNnCN}), we have $\#J_0(N)(\QQ)_\tor=\nC_N$. Thus, since $\dC_N$ and $\nH_N$ are coprime to $\nC_N$ by construction (cf.\ (\ref{eqn:prp-cln:dHNdCN}), (\ref{eqn:res-opt:nHNnCN})), we conclude from (\ref{eqn:res-art:N=arb-CNDmodnN}) 
that 
$\nH_N\dC_N\HGH(D)$ is not divisible by $\nC_N$, and $C_{\varphi_N}(D)$ is not divisible by $\nC_N$ either. In particular, $C_{\varphi_N}(D)$ is not zero.
So there must be an eigenform $\varphi\in S_{2,1}(N)$, for the level $N$ Hecke algebra $\mathbf{T}_N$,
such that $C_\varphi(D)$ is non-zero, where $C_\varphi(s^2-4n)$ is the coefficient of $q^ny^s$ in the Fourier expansion of $\varphi$. 
As we have mentioned in the proof of Corollary \ref{cor:res-opt:tor}, the ${\bf T}_N$-modules $S_{2,1}(N)$ and $S_2(N)$ are isomorphic. 
Both spaces are spanned by newforms since $N$ is prime, so there is a uniquely determined newform $f\in S_2(N)$ corresponding to $\varphi$.

At this point we apply 
Theorem 5.7 of \cite{MR1637086}, taking $k=2$, $m=1$ and $M=N$ in loc.\ cit.,
to obtain that
\begin{gather}\label{eqn:res-art:GKZLval}
	\frac{|C_{\varphi}(D)|^2}{\langle \varphi,\varphi\rangle}
	=
	\frac{\sqrt{|D|}}{2\pi}
	\frac{L(f\otimes D,1)}{\langle f,f\rangle}
\end{gather}
(cf.\ Corollary 1 of \cite{MR783554}), where $f\otimes D$ is defined by requiring that 
\begin{gather}\label{eqn:res-art:fotimesD}
	(f\otimes D)(\tau) = \sum_n c_f(n)\left(\frac{D}{n}\right)q^n
\end{gather}
when
	$f(\tau)=
	\sum_n c_f(n)q^n$,
and the $L$-function special value $L(f\otimes D,1)$ may be defined by setting
\begin{gather}\label{eqn:res-art:LfotimesD1}
L(f\otimes D,1):={2\pi}\int_0^\infty (f\otimes D)(it){\rm d}t.
\end{gather}

Let $I_f$ be the annihilator of $f$ in the Hecke algebra ${\bf T}_N$. 
Then $A=J_0(N)/I_fJ_0(N)$ is an optimal quotient of $J_0(N)$ (i.e.\ $I_fJ_0(N)$ is connected), and we have $L(A\otimes D,s)=L(f\otimes D,s)$ by construction. So in particular, $L(A\otimes D,1)$ is not zero. The proof is completed by applying the main result of \cite{MR1036843}, which tells us that the group of rational points on $A\otimes D$ is finite unless $L(A\otimes D,1)$ vanishes.
\end{proof}

The conclusion of Theorem \ref{thm:res-art:N=arb} simplifies in the case that $\dim J_0(N)$ is one-dimensional (i.e., an elliptic curve). For example, taking $N=11$, so that 
$\# J_0(N)(\QQ)_\tor=5$
(cf.\ (\ref{eqn:res-opt:nHNnCN})), we obtain the following corollary, which appeared earlier in \cite{MR927171}. 
\begin{cor}\label{cor:res-art:N=11}
Suppose that $D<0$ is a fundamental discriminant such that $11$ is inert in the ring of integers of $\QQ(\sqrt{D})$, 
and the class number of $\QQ(\sqrt{D})$ is not divisible by $5$. 
Then the elliptic curve defined by
\begin{gather}\label{eqn:res-art:J0(11)D}
y^2 = x^3 - 13392D^2x -  1080432D^3	
\end{gather}
has only finitely many rational points.
\end{cor}
As noted in \S~\ref{sec:int-res} (see (\ref{eqn:int-res:DtwistJ011})), the elliptic curve defined by (\ref{eqn:res-art:J0(11)D}) is the $D$-twist of the modular Jacobian $J_0(11)$ (which is Elliptic Curve 11.a2 in \cite{lmfdb}). 
The reader may find some further analogues of Corollary \ref{cor:res-art:N=11} in \cite{MR1040993}. 

To conclude we underscore 
that the identity (\ref{eqn:res-art:N=arb-CNDmodnN}), which underpins the proof of Theorem \ref{thm:res-art:N=arb}, in turn depends upon the integrality of 
(\ref{eqn:res-opt:torexp2}), which underpins the proof of Theorem \ref{thm:res-opt:tor}. 
Thus it is that the arithmetic-geometric results of Theorem \ref{thm:res-art:N=arb} and Corollary \ref{cor:res-art:N=11} arise as consequences of the optimal module classification, 
Theorem \ref{thm:res-opt:tor},
that we establish in this work.

\clearpage


\setstretch{1.08}
\addcontentsline{toc}{section}{References}

\providecommand{\bysame}{\leavevmode\hbox to3em{\hrulefill}\thinspace}
\providecommand{\MR}{\relax\ifhmode\unskip\space\fi MR }
\providecommand{\MRhref}[2]{
  \href{http://www.ams.org/mathscinet-getitem?mr=#1}{#2}
}
\providecommand{\href}[2]{#2}

\end{document}